\documentclass[reqno,a4paper,11pt]{amsart}
\usepackage[british]{babel}
\usepackage{fullpage,hyperref,cleveref,amssymb,mathtools,ifthen,bm,comment,soul,csquotes}
\usepackage{lmodern}
\usepackage[shortlabels]{enumitem}
\usepackage[T1]{fontenc}
\usepackage[all]{hypcap}
\usepackage{mdframed}
\usepackage{array}
\usepackage{tikz}
\usepackage{wrapfig}
\usepackage{pgfplots}
\usepackage[normalem]{ulem}
\usepackage[export]{adjustbox}
\usetikzlibrary{patterns}
\usepackage{xcolor}
\usepackage{thm-restate}
\pgfplotsset{compat=1.18}

\usepackage[backend=biber,giveninits=true,doi=false,url=false,isbn=false,style=trad-plain,hyperref]{biblatex}
\bibliography{bibliography}

\newtheorem{theorem}{Theorem}[section]
\newtheorem{proposition}[theorem]{Proposition}
\newtheorem{lemma}[theorem]{Lemma}
\newtheorem{corollary}[theorem]{Corollary}
\newtheorem{claim}[theorem]{Claim}
\Crefname{claim}{Claim}{Claims}
\theoremstyle{definition}

\newtheorem{remark}[theorem]{Remark}
\newtheorem{question}[theorem]{Question}

\newtheorem{definition}[theorem]{Definition}

\newtheorem{property}{Property}
\renewcommand{\theproperty}{P\arabic{property}}
\Crefname{property}{Property}{Properties}
\creflabelformat{property}{#2(#1)#3}

\newcommand{\car}{\mathbin{\Box}}
\newcommand{\symdif}{\mathbin{\triangle}}

\newcommand{\Bin}{\operatorname{Bin}}
\newcommand{\Boot}{\operatorname{Boot}}
\newcommand{\dmin}{\delta(G)}
\newcommand{\dmax}{\Delta(G)}

\newcommand{\dist}{\operatorname{dist}}

\renewcommand{\le}{\leqslant}
\renewcommand{\leq}{\leqslant}
\renewcommand{\ge}{\geqslant}
\renewcommand{\geq}{\geqslant}
\renewcommand{\emptyset}{\varnothing}
\newcommand{\eps}{\varepsilon}

\providecommand\given{\nonscript\:\ifthenelse{\equal{\delimsize}{}}{\big\vert}{\delimsize\vert}\nonscript\:\mathopen{}}
\let\Pr\undefined
\DeclarePairedDelimiterXPP\Pr[1]{\mathbb{P}}{[}{]}{}{#1}
\DeclarePairedDelimiterXPP\Ex[1]{\mathbb{E}}{[}{]}{}{#1}
\DeclarePairedDelimiterXPP\Prs[1]{\mathbb{P}^*}{[}{]}{}{#1}
\DeclarePairedDelimiterXPP\Exs[1]{\mathbb{E}^*}{[}{]}{}{#1}
\DeclarePairedDelimiterXPP\Prp[1]{\mathbb{P}'}{[}{]}{}{#1}
\DeclarePairedDelimiterXPP\Prone[1]{\mathbb{P}_1}{[}{]}{}{#1}
\DeclarePairedDelimiterXPP\Prtwo[1]{\mathbb{P}_2}{[}{]}{}{#1}

\title{Universal behaviour of majority bootstrap percolation on high-dimensional geometric graphs}
\author{Maur\'icio Collares}
\address{Instituto de Matem\'atica, Estat\'istica e Ci\^encia da Computa\c{c}\~ao, Universidade de S\~ao Paulo, Rua do Mat\~ao 1010, 05508-090 S\~ao Paulo, Brazil}
\email{collares@ime.usp.br}
\author{Joshua Erde}
\address{School of Mathematics, University of Birmingham, Birmingham, B15 2TT, UK}
\email{j.erde@bham.ac.uk}
\author{Anna Geisler \and Mihyun Kang}
\address{Institute of Discrete Mathematics, Graz University of Technology, Steyrergasse 30, 8010 Graz, Austria}
\email{\{geisler, kang\}@math.tugraz.at}
\subjclass{60K35, 60C05 (Primary)}

\begin{document}
\begin{abstract}
  Majority bootstrap percolation is a monotone cellular automaton that can be thought of as a model of infection spreading in networks.
  Starting with an initially infected set, new vertices become infected once at least half of their neighbours are infected.
  The average case behaviour of this process was studied on the $n$-dimensional hypercube by Balogh, Bollob\'{a}s and Morris, who showed that there is a phase transition as the typical density of the initially infected set increases: For small enough densities the spread of infection is  typically local, whereas for large enough densities typically the whole graph eventually becomes infected.
  Perhaps surprisingly, they showed that the critical window in which this phase transition occurs does not contain $p=1/2$, and they gave bounds on its width on a finer scale.
  In this paper we consider the majority bootstrap percolation process on a class of high-dimensional geometric graphs which includes many of the graph families on which percolation processes are typically considered, such as grids, tori and Hamming graphs, as well as other well-studied families of graphs such as (bipartite) Kneser graphs, including the odd graph and the middle layer graph.
  We show similar quantitative behaviour in terms of the location and width of the critical window for the majority bootstrap percolation process on this class of graphs.
\end{abstract}
\maketitle
  
\section{Introduction}

\subsection{Motivation}

Bootstrap percolation is a process on a graph  which models the spread of an infection through a population.
This model was first considered by Chalupa, Leath and Reich \cite{Chalupa1979Grid} to describe and analyse magnetic systems, and has been widely used to describe other interacting particle systems, see \cite{JLMTZ17}.

More generally, in the physical sciences and in particular statistical physics, models in which a system evolves according to a set of `local' and homogeneous update rules are known as \emph{cellular automata}.
This includes the Ising model or lattice gas automata \cite{Hardy1973Gas, Ising1925}.
Bootstrap percolation is then an example of a \emph{monotone} cellular automaton -- the `state' (infected or uninfected) of each site can only change in one direction, and the update rule is \emph{homogeneous} (the same for each vertex) and \emph{local} (determined by the states of the neighbours).

These models were first analysed rigorously by Schonmann \cite{Schonmann1992CellularAutomata}.
A related model, introduced by Bollob\'as \cite{Bollobas1968Saturation}, is weak saturation, in which the edges of a graph $G$ get activated if they complete a copy of a fixed graph $H$.
Bootstrap percolation has been used to model various processes, from the zero-temperature Ising model in physics \cite{Fontes2002Ising} to neural networks \cite{Amini2010Neural} or opinion spreading processes in sociology \cite{Granovetter1978Sociology}.
For a survey on monotone cellular automata results related to bootstrap percolation see \cite{Morris2017CellularAutomata} and for physical applications \cite{Morris2017Bootstrap}.

The bootstrap percolation process  on a graph $G$ evolves in discrete time steps, which we call \emph{rounds}, starting with an initial set $A_0 \subseteq V(G)$ of infected vertices. In each round, new vertices become infected if at least a certain threshold $r$ of their neighbours have already been infected. Since infected vertices never recover, the set $A_i$ of vertices which are infected by the $i$-th round is non-decreasing. In other words, we have a sequence of sets $A_0 \subseteq A_1 \subseteq \ldots$ where $A_{i} = A_{i-1} \cup \left\{ v \in V(G) \;:\; \left|N(v) \cap A_{i-1}\right| \geq r \right\}$ for $i\in \mathbb N.$
We will refer to this process as \emph{$r$-neighbour bootstrap percolation}.
If the infection spreads to the entire population, i.e., if $\bigcup_{i=0}^{\infty} A_i = V(G)$, then the initial set $A_0$ is said to \emph{percolate} on $G$.

It is apparent that given a graph $G$ and an initially infected set $A_0$, the $r$-neighbour bootstrap percolation process is deterministic, and so each subset $A_0 \subseteq V(G)$ is either percolating or non-percolating.
Typically, given a graph or a family of graphs, one can ask about the \emph{worst case} or extremal behaviour -- what are the `minimal' percolating sets \cite{Bidgoli2021Hamming, dukes2023extremal, Morrison2018ExtremalHypercube} -- or about the \emph{typical} behaviour -- are `most' sets percolating or non-percolating? A natural way to frame the latter question is to look at the probability $\Phi(p,G)$ that a \emph{random} subset of a fixed density $p$ percolates.
More precisely, in {\em random} bootstrap percolation, given a graph $G$ and $p \in (0,1)$,  we let ${\bf A}_p$ be a $p$-random subset of $V(G)$ where each vertex in $V(G)$ is included in ${\bf A}_p$ independently with probability $p$ and define
\[
  \Phi(p,G) := \Pr{{\bf A}_p \text{ percolates on } G}.
\]
Since the $r$-neighbour bootstrap percolation process is monotone, it is clear that the probability of percolation $\Phi(p,G)$ is monotonically increasing in $p$.
It follows from standard results on the existence of thresholds for monotone graph properties \cite{Friedgut1999MonotoneThreshold} that there is a \emph{threshold} for random bootstrap percolation -- the percolation probability transitions continuously from almost $0$ to almost $1$ in a very small window around the \emph{critical probability}
\[
  p_c(G) := \inf \left\{ p \in (0,1) \;:\;  \Phi(p,G) \geq \frac{1}{2} \right\}.
\]

In this paper we will consider this process on a family of graphs $(G_n)_{n \in \mathbb{N}}$, where the critical probability $p_c(G_n)$ is then a function of $G_n$.
Asymptotically, this transition from percolation almost never to almost surely then happens in some \emph{critical window}, an interval $I_n \subseteq [0,1]$ below which $\Phi(p,G_n) = o_n(1)$ and above which $\Phi(p,G_n) = 1 - o_n(1)$. We say there is a \emph{sharp threshold} if the width of the critical window is small in comparison to the critical probability, i.e., if the quotient satisfies
\[
  \lim_{n\to\infty} \frac{|I_n|}{p_c(G_n)} = 0.
\]

In the first appearance of bootstrap percolation \cite{Chalupa1979Grid} the authors gave some bounds for the location and width of the critical window in bootstrap percolation on the \emph{Bethe lattice}.
Subsequent work in the area of physics also focused on lattice-like structures and observed threshold phenomena \cite{AizenmanS2001Threshold}.
The simplest examples of lattices, and the most well-studied in terms of bootstrap percolation, come from finite grids $[n]^d$ and in particular the \emph{$d$-dimensional hypercube} $[2]^d$, which we will denote by $Q_d$.
Holroyd \cite{holroyd2002sharp} demonstrated the existence of a sharp threshold in $2$-neighbour bootstrap percolation in two-dimensional grids $[n]^2$, i.e., the case $r=d=2$ and for arbitrary $n\in \mathbb N$.
In three dimensions, bounds on the critical probability for $r=3$ were determined by Cerf and Cirillo \cite{Cerf1998Lattice3D}.
Later work of Balogh, Bollobás and Morris \cite{Balogh2003Sharp,BaBoMo2010HighDimGrid,BaBoMo20093D} showed that there is a sharp threshold in grids in arbitrary dimension.
Cerf and Manzo \cite{CerfManzo2002FinVol} gave the critical probability for all finite grids with $3 \leq r \leq d$ up to constant factors.
This line of work culminated in the breakthrough result of Balogh, Bollobás, Duminil-Copin and Morris \cite{balogh2011Grid}, which gave a sharp threshold in all dimensions:
\begin{theorem}[{\cite[Theorem 1]{balogh2011Grid}}]
Let $d \geq r \geq 1$ and consider the random $(r+1)$-neighbour bootstrap percolation process. Then
  \[
    p_c\left([n]^d\right)= \left(\frac{\lambda(d, r)+o(1)}{\log_{r}(n)}\right)^{d-r},
  \]
where $\lambda(d, r)$ is an explicit constant depending on $d$ and $r$, and $\log_r$ is the iterated logarithm. 
\end{theorem}
Here, and throughout the paper, our asymptotics will be in terms of the parameter $n$ unless otherwise stated.

As well as in lattice-like graphs, the typical behaviour of bootstrap percolation has been considered on many other graph classes such as binomial random graphs \cite{Janson2012Gnp, Kang2016Gnp}, inhomogeneous random graphs \cite{Foun2017Inhom} and infinite trees \cite{balogh2005infiniteTrees}.

In many physical applications there is a related model which is also natural to consider, where a vertex becomes infected once at least half of its neighbours have been infected \cite{Fontes2002Ising}.
We call this process the \emph{majority bootstrap percolation process}.
Note that, in general, this update rule is no longer homogeneous, as we have a different infection threshold $r(v)=d(v)/2$ for each vertex $v$.
However, for $d$-regular graphs, this is again an example of $r$-neighbour bootstrap percolation with $r=d/2$. Note that, when analysing the asymptotic behaviour of this process on a family of regular graphs $(G_n)_{n \in \mathbb{N}}$, even though the infection threshold $r$ is fixed for each $G_n$, it might vary as a function of $n$.

It is not hard to see that, for any $n$-regular graph $G$ whose order is sub-exponential in $n$, there is a sharp threshold for percolation in the random majority bootstrap percolation process near the point $p=\frac{1}{2}$.
Indeed, by standard concentration results, if $p > \frac{1}{2} + \varepsilon$, then with probability tending to one as $n \to \infty$ more than half of the neighbours of \emph{every} vertex in $G$ will lie in the initially infected set, and so the process will percolate after a single round.
Conversely, if $p < \frac{1}{2} - \varepsilon$, then with probability tending to one as $n \to \infty$ less than half the neighbours of \emph{every} vertex in $G$ will lie in the initially infected set, and so the process will stabilise with the infected set being the initial set, which is most likely not the entire vertex set.

Whilst this behaviour cannot be \emph{universal} to all $n$-regular graphs (consider for example a disjoint union of many copies of a fixed $n$-regular graph), it was shown by Balogh, Bollobás and Morris \cite{BaBoMo2009} that this sharp threshold occurs as well for graphs of order super-exponential in $n$ which satisfy certain nice structural properties. The key point here is to have some control on the \emph{neighbourhood expansion} of the host graph, to ensure that all vertices at distance $i$ from a fixed vertex $x$ have `many' neighbours at distance $i+1$ from $x$.
\begin{theorem}[{\cite[Theorem 2.2]{BaBoMo2009}}, informal] \label{thm:BaBoMogeneral}
  Let $G$ be a regular graph whose order is sufficiently small in terms of the neighbourhood expansion of the graph and consider the random majority bootstrap percolation process. Then 
    \[
    p_c(G)=\frac{1}{2}+o(1).
    \]
\end{theorem}

Furthermore, in the case of the hypercube, they looked at the location and width of the critical window on a finer scale, and found that the point $p=\frac{1}{2}$ actually lies outside, and in fact above, the critical window.

\begin{theorem}[{\cite[Theorem 2.1]{BaBoMo2009}}]\label{t:BBM}
  Let $Q_n$ be the $n$-dimensional hypercube and consider the random majority bootstrap percolation process where the set of initially infected vertices is a $p$-random subset ${\bf A}_p$ with
  \[
    p:=\frac{1}{2}-\frac{1}{2} \sqrt{\frac{\log n}{n}} + \frac{\lambda \log \log n}{\sqrt{n \log n}}.
  \]
  Then
  \begin{equation*}
    \lim_{n \to \infty} \Phi(p, Q_n) = \left\{
    \begin{array} {c@{\quad \textup{if} \quad}l}
      0 & \lambda \leq -2,       \\[+1ex]
      1 & \lambda > \frac{1}{2}.
    \end{array}\right.
  \end{equation*}
\end{theorem}

Since in this context it makes sense to consider the critical probability as having an `expansion' around $p=\frac{1}{2}$, it will be useful to introduce some definitions so as to be able to talk about the strength of bounds on the location and width of the critical window.

Suppose a property has a critical window of the form $I = p_1+p_2 + \ldots \pm p_k$ where $p_i = o_n(p_{i-1})$ for each $2 \leq i < k$ and $p_k = O_n(p_{k-1})$.
Then, whenever $p_i = o_n(p_{i-1})$ we say that $p_i$ is the \emph{$i$-th term} in the expansion of the critical probability and the \emph{width} of the critical window is at most $p_k$.
In this way, \Cref{t:BBM} determines the first two terms in the expansion of the critical probability to be $p_1=\frac{1}{2}$ and $p_2=-\frac{1}{2}\sqrt{\frac{\log n}{n}}$, and bounds the width of the critical window to be $O\left(\frac{\log\log n}{\sqrt{n \log n}}\right)$ for random majority bootstrap percolation on the $n$-dimensional hypercube.

\subsection{Main Results}
The aim of this paper is to investigate the random majority bootstrap percolation process on \emph{high-dimensional geometric} graphs.
The $n$-dimensional hypercube has many equivalent representations: For example, it is the nearest neighbour graph on the set of points $\{0,1\}^n$; it is the Cayley graph of $\mathbb{Z}_2^n$ with the standard generating set; it is the Cartesian product of $n$ copies of a single edge; and it is the skeleton of a particularly simple convex polytope in $\mathbb{R}^n$ -- all of which in some sense witness the fact that $Q_n$ is high-dimensional.

We will study the majority bootstrap percolation process on other graphs with similar high-dimensional representations, see \Cref{s:examples} for the precise treatment of the following examples.
For example, many natural classes of lattice-like graphs can be realised as Cartesian products of small graphs, such as grids, tori and Hamming graphs, which are the Cartesian products of paths, cycles and complete graphs, respectively.
Such graphs are commonly studied in percolation theory \cite{Ajtai1982ThresholdCube, Diskin2022Products, ErdosSpencer1979Hypercube, Heydenreich2007Tori} and in particular in the context of bootstrap percolation \cite{Bidgoli2021Hamming,KMM24}.

Another well-studied class of high-dimensional graphs are Kneser graphs and bipartite Kneser graphs, which encode how different subsets of a fixed set intersect.
Of particular interest here are the middle layers graph $M_n$ and the odd graph $O_n$ whose combinatorial properties have been extensively studied \cite{B79,M16,KT88}.
The odd graph $O_n$ belongs to the family of \emph{generalised odd graphs} that are related to incidence geometries and in particular to \emph{near polygons}, see \cite{BCN89}.
Another well-known example of a generalised odd graph is the $(n-1)$-dimensional \emph{folded hypercube} $\tilde{Q}_n$ obtained by joining each pair of antipodal points in the hypercube $Q_{n-1}$.
This graph has also commonly been used as a lower-diameter alternative to the hypercube in the context of network topology \cite{EAL91}.
Percolation processes have also been considered on some of these graphs, in the context of the transference of combinatorial results to sparse random substructures \cite{BBN15}.

In this paper we will study a general class $\mathcal{H}=\bigcup_{K \in \mathbb{N}} \mathcal{H}(K)$ of graphs satisfying certain structural properties (see \Cref{s:conditions} for formal definition), which are satisfied by various high-dimensional geometric graphs, and in particular the graphs listed above. Roughly speaking, $\mathcal{H}(K)$ consists of graphs whose structure is \emph{close} -- in a quantitative manner in terms of the parameter $K$ -- to being controlled by some coordinate system, which may only be defined locally with respect to each vertex and need not be globally coherent, where edges correspond to changing a single coordinate.

We will show that there is a certain \emph{universal behaviour} to the random majority bootstrap percolation process on graphs in the class $\mathcal{H}(K)$, which is controlled in some way by the degree sequence, and we will in particular determine the first few terms in the expansion of the critical probability for regular graphs in $\mathcal{H}(K)$.  
In fact, we show that the critical window does not contain $p=\frac{1}{2}$, for \emph{arbitrary} graphs $G$ in the class $\mathcal{H}(K)$, where the upper and lower boundary
of the critical window are functions of the minimum and maximum degree of $G$, denoted by $\delta(G)$ and $\Delta(G)$, respectively.

Our main theorem is as follows.
\begin{theorem}\label{t:mainThm}
  Let $(G_n)_{n \in \mathbb{N}}$ be a sequence of graphs in $\mathcal{H}(K)$ for some $K\in \mathbb{N}$ such that $\delta(G_n) \to \infty$ as $n \to \infty$ and let $p=p(n) \in [0,1]$. Consider the random majority bootstrap percolation process where the set of initially infected vertices is a $p$-random subset ${\bf A}_p$. 
  Then for any constant $\eps > 0$,
  \begin{equation*}
    \lim_{n \to \infty} \Phi(p, G_n) = \left\{
    \begin{array} {c@{\quad \textup{if} \quad}l}
      0 & p < \frac{1}{2}-\left(\frac{1}{2}+\eps\right) \sqrt{\frac{\log \delta(G_n)}{\delta(G_n)}}, \\[+1ex]
      1 & p > \frac{1}{2}-\left(\frac{1}{2}-\eps\right) \sqrt{\frac{\log \Delta(G_n)}{\Delta(G_n)}}.
    \end{array}\right.
  \end{equation*}
\end{theorem}

In particular, if the graphs $G_n$ are regular, then \Cref{t:mainThm} determines the first two terms in the expansion of the critical probability, as in \Cref{t:BBM} for the hypercube.

\begin{corollary}\label{c:mainThm}
  Let $(G_n)_{n \in \mathbb{N}}$ be a sequence of $d$-regular graphs in $\mathcal{H}(K)$ for some $K \in \mathbb{N}$ such that $d=d(n) \to \infty$ as $n \to \infty$ and let $p=p(n) \in [0,1]$.  Consider the random majority bootstrap percolation process where the set of initially infected vertices is a $p$-random subset ${\bf A}_p$.  
  Then for any constant $\eps > 0$,
  \begin{equation*}
    \lim_{n \to \infty} \Phi(p, G_n) = \left\{
    \begin{array} {c@{\quad \textup{if} \quad}l}
      0 & p < \frac{1}{2}-\left(\frac{1}{2}+\eps\right) \sqrt{\frac{\log d}{d}}, \\[+1ex]
      1 & p > \frac{1}{2}-\left(\frac{1}{2}-\eps\right) \sqrt{\frac{\log d}{d}}.
    \end{array}\right.
  \end{equation*}

  In particular, this holds if $G_n$ is
  \begin{itemize}
    \item an $n$-dimensional regular Cartesian product graph whose base graphs have bounded size;
    \item the $n$-dimensional middle layer graph $M_n$;
    \item the $n$-dimensional odd graph $O_n$; or
    \item the $n$-dimensional folded hypercube $\tilde{Q}_{n}$.
  \end{itemize}
\end{corollary}
Note that $M_n$, $O_n$ and $\tilde{Q}_{n}$ are $n$-regular. 

\Cref{c:mainThm} shows that for \emph{$d$-regular} graphs in $\mathcal{H}(K)$ there is a sharp percolation threshold for the random majority bootstrap percolation process. More precisely, we find that the behaviour in \Cref{t:BBM}, when scaled correctly, is in some way \emph{universal} for these high-dimensional graphs, in that the first two terms in the expansion of the critical probability are 
$\frac{1}{2}-\frac{1}{2} \sqrt{\frac{\log d}{d}}$ and the critical window has width $o\left(\sqrt{\frac{\log d}{d}}\right)$: In other words, for any   \emph{$d$-regular} graph $G$ in $\mathcal{H}(K)$, the critical 
probability satisfies
$$ p_c(G) = \frac{1}{2}-\frac{1}{2} \sqrt{\frac{\log d}{d}} + o\left(\sqrt{\frac{\log d}{d}}\right).$$

\subsection{Proof Techniques}

In \cite{BaBoMo2009} Balogh, Bollob\'{a}s and Morris show that, at least in the $n$-dimensional hypercube, the event that a vertex becomes infected is in some sense `locally determined'.
Indeed, for the $1$-statement they show that with probability tending to one as $n \to \infty$ (whp) every vertex becomes infected after at most eleven rounds.
For the $0$-statement they analyse a related process, which stochastically dominates the random majority bootstrap percolation process, and show that whp this process stabilises after three rounds, and in fact whp a positive proportion of the vertices remains uninfected.
Note that whether or not a vertex $v$ is infected after $k$ rounds only depends on the set of initially infected vertices at distance at most $k$ from $v$.

A key tool in our proof of \Cref{t:mainThm} is that the degree distributions of graphs in $\mathcal{H}(K)$ are locally quite `flat', in the sense that the degrees in a small neighbourhood of a vertex $v$ are relatively close to the degree of $v$.
Roughly, such graphs look \emph{locally approximately regular}.
If, as in the work of Balogh, Bollob\'{a}s and Morris \cite{BaBoMo2009}, we expect the probability that a vertex $v$ becomes infected to be `locally determined', we might hope that for each vertex $v$ there is a critical probability $\tilde p_c(v)$, which depends only on its degree $d(v)$, such that significantly above this threshold it is very likely that $v$ becomes infected, whereas significantly below this threshold there is a positive probability that $v$ remains uninfected, at least for the first few rounds of the process.
In particular, we should expect the process to percolate when $p \gg \max_{v \in V(G)} \tilde p_c(v)$, and we should expect the process not to percolate when $p \ll \min_{v \in V(G)} \tilde p_c(v)$.

Let us describe in more detail how we make this heuristic argument precise.
For the $1$-statement we note that, since we are considering probabilities close to $\frac{1}{2}$, the probability that a vertex is not infected in the first round is roughly $\frac{1}{2}$.
We first show that, for each $v \in V(G)$ when $p$ is above $\tilde p_c(v):= \frac{1}{2}-\frac{1}{2}\sqrt{\frac{\log d(v)}{d(v)}}$, the probability that $v$ is \emph{not infected} after two rounds is already much smaller, at most $\frac{1}{4}$ (\Cref{l:constant}).

We then bootstrap this result twice, first to show that the probability that $v$ is not infected after five rounds is exponentially small in $d(v)$ (\Cref{l:dimred}), and then again to show that the probability that $v$ is not infected after eleven rounds is super-exponentially small in $d(v)$ (\Cref{l:super-exponential}).
Since we assume that the graphs in $\mathcal{H}(K)$ are not too large, roughly of exponential order, this is sufficient to deduce the $1$-statement using a union bound.
To avoid dependencies when bootstrapping our results, and to simplify our analysis, we actually analyse a slightly more restrictive bootstrap percolation process, which is dominated by the original process, where the local infection parameter $r(v)$ is slightly larger than $\frac{d(v)}{2}$. This broad strategy for proving the 1-statement follows the approach of~\cite{BaBoMo2009}.

A key property of graphs in $\mathcal{H}(K)$ which allows us to bootstrap these results is a certain ``fractal self-symmetry'' which roughly says we can split those graphs into many copies of (lower-dimensional) graphs having the same structural properties.
This property has also been key to the study of bond and site percolation on high-dimensional graphs \cite{CDE24,DK23, Diskin2022Products}.

For the $0$-statement, we analyse a slight variant of the majority bootstrap process which dominates the original process, where the infection parameter $r$ varies over time, starting slightly below $\frac{d(v)}{2}$ and increasing to $\frac{d(v)}{2}$ in a finite number of rounds. This idea, which originates in \cite{BaBoMo2009}, allows us to assume that vertices which become infected in round $i$ for $i > 1$ have a significant number of neighbours which became infected in round $i-1$. We first show that in this new process, a vertex which is not infected by the second round is extremely unlikely to become infected in the third round, and so whp this process stabilises after two rounds (\Cref{l:inA3-A2}). Here, a judicious choice of parameterisation simplifies the analysis of this process, allowing us to conclude that the process stabilises after only two rounds.
We then conclude by showing that whp some vertices are not infected by the second round (\Cref{l:inA2-A1}).

The paper is structured as follows:
In \Cref{s:preliminiaries} we introduce some notation and probabilistic tools that we use.
In \Cref{s:conditions} we explicitly describe the structural assumptions satisfied by our class  $\mathcal{H}(K)$ of high-dimensional graphs.
The proof of \Cref{t:mainThm} is then split up into a proof of the $1$-statement in \Cref{s:1-statement} and a proof of the $0$-statement in \Cref{s:0-statement}, with some of the technical details deferred to Sections \ref{sec:11suffice} and \ref{sec:dominating}.
Afterwards we give specific examples of several graphs that are contained in the class $\mathcal{H}$ of graphs in Section \ref{s:examples}.
We close with a discussion of the limits of the techniques used and some open questions in \Cref{s:discussion}.

\section{Preliminaries}\label{s:preliminiaries}

For $n \in \mathbb{N}$ we denote by $[n]$ the set of integers up to $n$, i.e., $[n]:= \{1, \dots, n\}$.
Given $y,z \in \mathbb{R}$ we will write $ y \pm z$ to denote the interval $[y-z,y+z]$. Similar to $O(\cdot)$ notation, an inclusion is meant whenever $\pm$ appears in the context of an equality. For example, $x = y \pm z$ means that $x \in [y-z, y+z]$. 
Whenever the base of a logarithm is not specified, we use the natural logarithm, i.e., $\log x = \ln x =\log_e x$. For ease of presentation we will omit floor and ceiling signs in calculations.


\subsection{Probabilistic tools}

In this section we will state a number of probabilistic tools which are used throughout the paper.

We will assume knowledge of Markov's and Chebyshev's inequalities as standard (see~\cite{Alon2016Book} for basic probabilistic background).
However, throughout the paper we will need much more precise tail bounds, both from above and below, for binomial and related distributions.
The first is a standard form of the Chernoff bounds, see for example~\cite[Appendix A]{Alon2016Book}.
\begin{lemma}\label{l:Chernoff}
  Let $d \in \mathbb{N}$, $0 < p < 1$, and $X \sim \Bin(d,p)$. Then
  \begin{enumerate}[$(a)$]
    \item For every $t \geq 0$,
      \[ \Pr[\big]{|X - dp| \geq t} \; \le \; 2\exp\left( -\frac{2t^2}{d} \right); \]
    \item\label{i:bigtail} For every $b \geq 1$,
      \[ \Pr{X \geq bdp}  \leq \left(\frac{e}{b}\right)^{bdp}. \]
  \end{enumerate}
\end{lemma}

At times we will also need an anti-concentration result for binomial random variables.
As these converge in distribution to a normal distribution we will derive these from tail bounds for the standard normal distribution $N(0, 1)$.
Proofs can be found in \cite[Section 5.6]{Asymptopia}.

\begin{lemma}\label{l:CLT}
  For $p=p(d)$ bounded away from $0$ and $1$, and $f(d) = o(d^{1/6})$, it holds that
  \[\Pr*{\frac{\Bin(d, p)-d p}{\sqrt{dp(1-p)}} \geq f(d)} = (1+o(1)) \Pr*{N(0, 1) \geq f(d)}. \] Moreover, if $f(d) \to \infty$ as $d \to \infty$, the probability that the standard normal distribution exceeds $f(d)$ satisfies
  \[\Pr{N(0,1) \geq f(d)} = \frac{1+o(1)}{f(d) \sqrt{2 \pi}} \exp\left(-\frac{f(d)^2}{2}\right).\]
\end{lemma}

Since binomial random variables can be written as sums of independent Bernoulli random variables, the standard inequality due to Hoeffding~\cite{Hoeffding1963} readily implies the following lemma.
\begin{lemma}\label{l:sumBin}
  Let $k,d_1,\ldots,d_k \in \mathbb{N}$ and $p \in (0,1)$.
  Let $X_i$ be independent random variables with $X_i \sim \Bin(d_i,p)$ for each $i \in [k]$, let $Y = \sum_{i=1}^k iX_i$, and let $D(k) = \sum_{i=1}^k i^2 d_i$.
  Then, for every $\tau > 0$,
  \[\Pr[\big]{Y \ge \Ex{Y} + \tau} \le \exp\left( - \frac{2\tau^2}{D(k)} \right) \le \exp\left( - \frac{2p\tau^2}{k\cdot\Ex{Y}} \right).\]
\end{lemma}

We will also need the following lemma from \cite{BaBoMo2009} which broadly tells us that, if we think of a $\Bin(n,p)$ random variable as the number of successful trials in a sequence of $n$ independent trials with success probability $p$, then there is only a slight correlation between the result and the outcome of the first trial. 

\begin{lemma}[\cite{BaBoMo2009}, Lemmas 3.8 and 3.9]\label{l:Binvariant}
Let $d \in \mathbb{N}$, let $p =p(d)= \frac{1}{2} - o(1)$. Let $X_1,\ldots, X_d$ be independent and identically distributed $\textup{Ber}(p)$ random variables and let $X=\sum_{i=1}^d X_i \sim \Bin(d, p)$.
  Then, for any $0 \le m = m(d) \le d/2$,
  \[\Pr*{X \ge m \given X_1=1} = \big( 1 + o(1) \big)\,\Pr[\big]{X \ge m}\]
  and
  \[ \Pr*{X_1 = 1 \given X \ge m} = \big(1 + o(1)\big)\,\Pr[\big]{X_1 = 1}.\]
\end{lemma}

 By induction and the law of total probability, it follows from \Cref{l:Binvariant} that for any bounded $k\in \mathbb N$ and any $\varepsilon=(\varepsilon_1, \ldots, \varepsilon_k) \in \{0, 1\}^k$,
 \[ \Pr[\big]{X \ge m \given X_1 = \varepsilon_1, \ldots, X_k = \varepsilon_k} = \big( 1 + o(1) \big)\,\Pr[\big]{X \ge m}. \]
 
 Finally we will need the fact that the median of a binomial random variable is essentially equal to its mean, see for example \cite[Theorem 1]{Kaas1980MedianBin}.
\begin{proposition}\label{p:MedianBin}
  Let $X \sim \Bin(d, p)$.
  Then the median of $X$ is either $\lfloor\Ex{X}\rfloor$ or $\lceil\Ex{X} \rceil$, and thus $\Pr{X \geq C} \geq \frac{1}{2}$ for any integer $C$ such that $C \leq \Ex{X}$.
\end{proposition}

\section{Nice structural properties of geometric graphs}\label{s:conditions}

In this section we discuss the properties of the graphs in the class $\mathcal{H}(K)$. Roughly, one can think of these graphs  as having a structure which is in some sense \emph{close} to being governed by some \emph{local coordinate system}, where edges are only allowed between vertices which differ in a single coordinate. Crucially there are only ever a bounded number of neighbours in each coordinate.

Given a graph $G$ we will write $d(G)$ for the average degree of $G$, $\delta(G)$ for the minimum degree of $G$ and $\Delta(G)$ for the maximum degree of $G$. Given two vertices $x,y \in V(G)$ we will write $\dist_{G}(x,y)$ for the \emph{distance} between $x$ and $y$ in $G$, that is, the length of the shortest $x-y$ path in $G$. Given $k \in \mathbb{N} \cup \{0\}$ and $x \in V(G)$ we let 
\[B_G(x,k):= \{ y \in V(G) \;:\; \dist_G(x,y) \leq k \}\]
be the \emph{ball of radius $k$} centred at $x$ and
\[S_G(x,k) := B_G(x,k) \setminus B_G(x,k-1)\]
be the \emph{sphere of radius $k$} centred at $x$, where $S_G(x, 0) := \{x\}$. Note that $B_G(x, 0) = \{x\}$, $S_G(x, 1) = N_G(x)$, and $B_G(x, 1) = \{x\} \cup N_G(x)$. When the underlying graph $G$ is clear from the context, we will omit the subscript in this notation.

Given $K \in \mathbb{N}$ define the \emph{class $\mathcal{H}(K)$ of geometric graphs}  recursively by taking the single-vertex graph $K_1 \in \mathcal{H}(K)$ and then taking every graph $G$ satisfying \Cref{c:regular,c:backwards,c:localconnection,c:projection,c:partitioncond,c:bound} below.

\begin{property}[Locally almost regular]\makeatletter\edef\@currentlabel{(\theproperty)}\makeatother\label{c:regular}
 For every $x \in V(G)$, $\ell \in \mathbb{N}$ and $y \in S(x, \ell)$, we have $$|d(x)-d(y)| \leq K \ell.$$
\end{property}
In this case we say that $G$ is \emph{locally $K$-almost regular}. Heuristically, vertices at a small distance should agree in almost all their coordinates, and so their degrees, which are the sum of the number of neighbours in each coordinate, should be similar.

\begin{property}[Bounded backwards expansion]\makeatletter\edef\@currentlabel{(\theproperty)}\makeatother\label{c:backwards}
 For every $x \in V(G)$, $\ell \in \mathbb{N}$ and $y \in S(x, \ell)$, we have 
  $$|N(y) \cap B(x, \ell)| \leq K\ell.$$
\end{property}
In this case we say that $G$ has \emph{$K$-bounded backwards expansion}.
Note that having bounded backwards expansion is qualitatively equivalent to having good \emph{neighbourhood expansion} as in \Cref{thm:BaBoMogeneral} from \cite{BaBoMo2009}.
Heuristically, a vertex $y$ at distance $\ell$ from $x$ differs from $x$ in at most $\ell$ coordinates, and the only neighbours of $y$ which are not further away from $x$ also differ in a subset of these coordinates. Note that \Cref{c:regular,c:backwards} imply that $$|S(x, \ell)|= \Theta(d(x)^\ell).$$

\begin{property}[Typical local structure]\makeatletter\edef\@currentlabel{(\theproperty)}\makeatother\label{c:localconnection}
For every $x\in V(G)$ there is a set $D(x) \subseteq V(G) \setminus \{x\}$ of \emph{non-typical} vertices such that for every $\ell \in \mathbb{N}$ the following hold: 
  \begin{enumerate}[$(i)$]
    \item\label{i:small} $|D(x)\cap S(x,\ell)| \leq  K^{\ell-1} d(x)^{\ell-1}$;
    \item\label{i:sparse}  
    $|D(x)\cap N(y)| \leq K\ell$ for every vertex $y \in S(x, \ell) \setminus D(x)$;
    \item\label{i:cherry} every two vertices in $S(x, \ell) \setminus D(x)$ have at most one common neighbour in $S(x, \ell+1) \setminus D(x)$.
  \end{enumerate}
\end{property}
For every graph $G$ that fulfils \Cref{c:localconnection} and every $x \in V(G)$, we set 
\[S_0(x, \ell) := S(x, \ell) \setminus D(x).\]

Heuristically, since the number of neighbours in each coordinate is always bounded, a typical vertex at distance $\ell$ from $x$ differs in precisely $\ell$ coordinates from $x$. The set $D(x)$ takes care of the non-typical vertices, as well as taking into account the `error' in our approximate coordinate system. Conditions \ref{i:small} and \ref{i:sparse} say that the set $D(x)$ is small and locally sparse, whereas Condition \ref{i:cherry} clarifies a property we should expect of $S_0(x, \ell)$ -- if two vertices $u$ and $w$ differ from $x$ in precisely $\ell$ coordinates there is at most one vertex $v$ which differs from $x$ in precisely $\ell+1$ coordinates and differs from $u$ and $w$ in a single coordinate.

\begin{property}[Projection]\makeatletter\edef\@currentlabel{(\theproperty)}\makeatother\label{c:projection}
  For every $x \in V(G)$, $\ell \in \mathbb{N}$ and $y \in S(x, \ell)$, there is a subgraph $G(y)$ of $G$ such that the following hold: 
  \begin{enumerate}[$(i)$]
    \item\label{i:P4inclusion} $y \in V(G(y))$;
    \item $G(y) \in \mathcal{H}(K)$;
    \item\label{i:P4intersection} $V(G(y)) \cap B(x, \ell-1) = \emptyset$;
    \item\label{i:P4degree} $ |d_{G(y)}(w) - d_{G}(w)| \leq K \ell$ for all $w \in V(G(y))$.
      \end{enumerate}
\end{property}
In this case we say that $G$ has the \emph{$K$-projection property}. Heuristically, the idea here is that vertices at a short distance will only differ in a small number of `coordinates', and by fixing these coordinates and allowing the others to vary, we should obtain a lower-dimensional graph $G(y)$ with similar properties to $G$ which is disjoint from $B(x,\ell-1)$. This property reflects a notion of fractal self-symmetry of the graph.

We note that in all our applications, the graphs $G(y)$ can in fact be taken to be a `lower dimensional' graph with a similar description to $G$. For example, in the case of the hypercube these projections can be taken to be subcubes, although in general they might come from a slightly more general family than the original graph $G$.

\begin{property}[Separation]\makeatletter\edef\@currentlabel{(\theproperty)}\makeatother\label{c:partitioncond}
For every $x \in V(G)$, $\ell \in \mathbb{N}$ and $y \in S_0(x, \ell)$, we have
  \[
    |B(y, 2\ell-1) \cap S_0(x, \ell)| \leq \ell K^{\ell}d(x)^{\ell-1}.
  \]
\end{property}
In this case we say that $G$ satisfies the \emph{$K$-separation property}. Heuristically, a vertex $y \in S_0(x,\ell)$ will differ from $x$ in $\ell$ coordinates and, for each such $y$, the vertices in $B(y, 2\ell-1) \cap S_0(x, \ell)$ will also differ from $y$ in one of these $\ell$ coordinates in which $x$ differs from $y$, leaving one fewer `degree of freedom' for the choice of such vertices.

In particular, the $K$-separation property implies the following separation lemma.
\begin{lemma}[Separating partition]\label{l:partitiondist}
  Suppose $G$ satisfies \Cref{c:localconnection,c:partitioncond} for some $K \in \mathbb{N}$. For each $x \in V(G)$ and $\ell \in \mathbb{N}$, there exists a partition $\mathcal{P}$ of $S(x, \ell)$,
  \[ S(x, \ell) = P_1 \cup \dots \cup P_m, \]
  into $m$ disjoint sets $P_1, \ldots, P_m$ such that  the following hold:
    \begin{enumerate}[$(i)$]
    \item $m \leq  (\ell+1) K^{\ell} d(x)^{\ell-1}$;
    \item $\dist(y_1, y_2) \geq 2\ell$ for all distinct $y_1, y_2 \in P_j$ and every $j \in [m]$.
    \end{enumerate}
\end{lemma}
\begin{proof}
  By \Cref{c:localconnection} there are at most $K^{\ell-1} d(x)^{\ell-1}$ vertices in $D(x)\cap S(x,\ell)$. Choose a new partition class $P_i$ for each such vertex.

  Choose the partition of $S_0(x, \ell)$ greedily by putting each vertex $w \in S_0(x, \ell)$ in the first partition class such that the distance condition is satisfied.
  For $y \in S_0(x, \ell)$, by \Cref{c:partitioncond} we have $|B(y, 2\ell-1) \cap S_0(x, \ell)| \leq \ell K^{\ell} d(x)^{\ell-1}$.
  Thus we get at most this many partition classes for $S_0(x, \ell)$.

  Combining the number of partition classes from $S(x, \ell) \cap D(x)$ and $S_0(x, \ell)$ proves the claim.
\end{proof}

We note that while \Cref{c:regular,c:backwards,c:localconnection,c:projection,c:partitioncond}  make assertions for each $\ell \in \mathbb{N}$, in fact in the proof of Theorem \ref{t:mainThm} we will only need these properties for $\ell \leq 6$. Finally, our proof methods require that the graph $G$ is not too large.
\begin{property}[Exponential order]\makeatletter\edef\@currentlabel{(\theproperty)}\makeatother\label{c:bound}
  $G$ has at most $\exp(K \delta(G))$ vertices.
\end{property}

\section{Proof of the 1-statement of Theorem \ref{t:mainThm}}\label{s:1-statement}
Throughout the rest of the paper we assume that $G$ is a graph in $\mathcal{H}(K)$ for some $K\in \mathbb{N}$  (so $G$ satisfies \Cref{c:regular,c:backwards,c:localconnection,c:projection,c:partitioncond,c:bound} for this $K$), and all asymptotics will be as $\delta(G) \to \infty$. 

In this section we will prove the 1-statement of \Cref{t:mainThm}:  For any $\varepsilon >0$, if
  \[
    p > \frac{1}{2}-\left(\frac{1}{2}-\eps\right) \sqrt{\frac{\log \dmax}{\dmax}},
  \]
then $$ \Phi(p,G):= \Pr{{\bf A}_p \text{ percolates on } G} \to 1.$$ In fact,  we will prove that 
$$
\Pr{ \exists x \in V(G): x \notin A_{11}} =o(1).
$$
Along the proof we will introduce auxiliary lemmas, but their proofs are deferred to \Cref{sec:11suffice}. 
 
For each vertex $x \in V(G)$ let \[ \sigma(x) := \sqrt{\frac{\log d(x)}{d(x)}}\]
and choose the probability $p$ to be significantly larger than 
\[\tilde p_c(x) := \frac{1}{2} - \frac{1}{2} \sigma(x).\]
More precisely, for $c < \frac{1}{2}$ we let $$p := \frac{1}{2} - c\sigma(x).$$
Since the initially infected set of vertices $A_0$ is distributed as a $p$-random subset ${\bf A}_p$ of $V(G)$, the probability that a vertex $x$ lies in $A_0$ is approximately $\frac{1}{2}$. We start by showing that the probability that $x$ is infected after the first two rounds is already significantly larger than $\frac{1}{2}$.

In fact, we will show a slightly stronger statement, bounding the probability that a vertex $x$ is infected in the first two rounds of a slightly weaker infection process, where the threshold to infect a vertex $v$ is raised from $\frac{d(v)}{2}$ to $\frac{d(v)}{2} + m$ for some fixed $m \in \mathbb{N} \cup \{0\}$.
More precisely, we define recursively $A_0(m): =A_0$ and for each $i \in \mathbb{N} \cup \{0\}$
\[
  A_{i+1}(m) := A_i(m) \cup \left\{ v \in V(G) \;:\; |A_i(m) \cap N(v)| \geq \frac{d(v)}{2} + m \right\}.
\]
Note that if $m=0$, then $A_i = A_i(m)$ for all $i \in \mathbb{N}$ and that for any $m \geq 0$, \[A_i(m) \subseteq A_i \text{ for all } i \in \mathbb{N}.\]

The event that a vertex $x$ is infected by the second round depends only on vertices at distance at most two from $x$. More precisely, whether $x$ is infected in the second round only depends on how many neighbours of $x$ are infected initially and how many are infected in the first round. The number of initially infected neighbours is given by a binomial random variable with expectation $d(x)/2- c\sqrt{d(x) \log d(x)}$, and so in order to infect $x$ in the second round, on average we need around $c\sqrt{d(x) \log d(x)}$ many additional neighbours of $x$ to be infected in the first round.

For each neighbour of $x$ there is a small but non-negligible chance that it is infected in the first round. Indeed, since the number of initially infected neighbours of a vertex is distributed binomially, we can use the anti-concentration statement in \Cref{l:CLT} to calculate quite precisely the probability that it is infected in the first round, which is around $\Omega\left(d(x)^{-2c^2}\right)$ as we will show later, and so the expected number of neighbours of $x$ which will be infected in the first round is of order $\Omega\left(d(x)^{1-2c^2}\right) \gg \sqrt{d(x) \log d(x)}$. Furthermore, the local structure of graphs in $\mathcal{H}(K)$ ensures that for two neighbours of $x$ the events that they get infected in the first round are close to independent. This allows us to show that it is reasonably likely that $x$ will be infected in the second round using a second moment argument.

\begin{restatable}{lemma}{lconstant}\label{l:constant}
  Let $m \in \mathbb{N} \cup \{0\}$, $x\in V(G)$, $c < \frac{1}{2}$, $p = \frac{1}{2} - c\sigma(x)$ and $A_0 \sim {\bf A}_p$.
  If there is a $K \in \mathbb{N}$ such that $G \in \mathcal{H}(K)$, then
  \[
    \Pr*{x \in A_2(m)} \geq \frac{3}{4} + o(1).
  \]
\end{restatable}
We note that for the proof of \Cref{l:constant} we will only need \Cref{c:regular,c:backwards,c:localconnection} (see \Cref{sec:11suffice}).

Next we show that, above the probability $\tilde p_c(x)=\frac{1}{2}-\frac{1}{2} \sigma(x)$, the probability that a vertex $x$ is not infected shrinks quickly, from a constant probability in the second round to an exponentially small one by the fifth round.

Let us sketch the main ideas. If a vertex $x$ is not infected by the fifth round, then there is a subset $T'\subseteq N(x)$ of size $\frac{d(x)}{2}$ which is not infected by the fourth round.
Since no vertex in $T'$ is infected by the fourth round, and each vertex in $T'$ has approximately degree $d(x)$ there are at most roughly $d(x)|T'|/2$ many edges from $T'$ to $A_3$.
However, using the explicit structure of our graph $G \in \mathcal{H}(K)$, it is easy to show that the number of edges from $T'$ to the sphere $S(x,2)$ is approximately $d(x)|T'|$.

Since, by \Cref{l:constant}, each vertex is in $A_2$ (and hence also in $A_3$) with probability larger than say $2/3$, it should be very unlikely that fewer than half the edges from $T'$ to $S(x,2)$ go to vertices in $A_3$. However, due to the dependencies between the events that vertices in $S(x,2)$ lie in $A_3$, it is difficult to make this precise.

Instead, we look one round further -- if many vertices in $T = N(T') \cap S(x,2)$ are not infected by the third round then we again find that at most half of the edges from $T$ to the sphere $S(x,3)$ go to vertices in $A_2$.
However, by \Cref{l:partitiondist} we can partition $S(x,3)$ into $O(d^2)$ sets whose pairwise distance is at least $6$.
In particular, in any partition class the events that the vertices lie in $A_2$ are mutually independent.
By a double-counting argument we can find one partition class $B \subseteq S(x, 3)$ in which less than half of the edges from $T$ to $B$ go to vertices in $A_2$, which is very unlikely since each vertex in $B$ lies in $A_2$ with probability at least $2/3$ and these events are independent.

However, whether the vertices in $B$ lie in $A_2$ might still depend on the sets of initially infected vertices in $S(x,1)$ and $S(x,2)$ which also influence our choices of $T',T$ and $B$.
To get around this issue we use \Cref{c:projection} to find for each $y\in B$ a slightly smaller graph $G(y) \in \mathcal{H}(K)$ with $G(y) \subseteq G \setminus B(x, 2)$, where the vertex degrees and structure are approximately preserved.
This guarantees that infection with a slightly increased threshold in $G(y)$ is sufficient to imply infection in $G$, and allows us to apply \Cref{l:constant} to $y$ inside the subgraph $G(y)$.

\begin{restatable}{lemma}{ldimred} \label{l:dimred}
  Let $x\in V(G)$, $c < \frac{1}{2}$, and $p = \frac{1}{2} - c\sigma(x)$.
  If there is a $K \in \mathbb{N}$ such that $G \in \mathcal{H}(K)$, then there exists a $\beta >0$ (independent of $x$) such that
  \[
    \Pr[\big]{ x \notin A_5 } \leq \exp\bigl( - \beta d(x)\bigr).
  \]
\end{restatable}
We note that for the proof of \Cref{l:dimred} we will use \Cref{c:regular,c:backwards,c:localconnection,c:projection,c:partitioncond} (see \Cref{sec:11suffice}).

Finally, we bootstrap \Cref{l:dimred} to show that after six more rounds the probability that a vertex is not infected shrinks even further, and becomes super-exponentially small.
Since the order of the graphs in \Cref{t:mainThm} are exponential in their vertex degrees, this will be enough to deduce that, above an appropriate threshold, whp all vertices will become infected by the eleventh round.

Here we can afford to be slightly less careful than in \Cref{l:dimred}.
If a vertex $x$ is not infected by the $k$-th round, it is relatively easy to show that the number of vertices at distance $\ell$ which are not infected by the $(k-\ell)$-th round must be growing like $\Theta\left(d(x)^{\ell}\right)$.
In particular, if $x$ is not infected by the eleventh round, there is some set $T'$ with $\Theta(d(x)^{6})$ many vertices in the sphere $S(x,6)$ which are uninfected by the fifth round.
By \Cref{l:partitiondist} we can partition $S(x,6)$ into $O(d(x)^5)$ many sets within which all vertices have pairwise distance at least twelve.
By an averaging argument, some partition class $P$ must contain a large number of vertices in $T'$. However, since the events that the vertices of $P$ are not infected by the fifth round are independent, and each is unlikely by \Cref{l:dimred}, it follows from Chernoff's inequality (\Cref{l:Chernoff}) that it is extremely unlikely that any partition class contains a large number of vertices in $T'$.

\begin{restatable}{lemma}{lsuperexp}\label{l:super-exponential}
  Let $x\in V(G)$, $c < \frac{1}{2}$, and $p = \frac{1}{2} - c\sigma(x)$.
  If there is a $K \in \mathbb{N}$ such that $G \in \mathcal{H}(K)$, then there exists a $\beta >0$ (independent of $x$) such that
  \[
    \Pr[\big]{x \notin A_{11}}< \exp\bigl(-\beta d(x)^2\bigr).
  \]
\end{restatable}
We note that for the proof of \Cref{l:super-exponential} we will use \Cref{c:regular,c:backwards,c:localconnection,c:projection,c:partitioncond} (see \Cref{sec:11suffice}).

We are now in a position to prove the $1$-statement of \Cref{t:mainThm}.
\begin{proof}[Proof of $1$-statement in \Cref{t:mainThm}]
Let    $K \in \mathbb{N}$ and let $(G_n)_{n \in \mathbb{N}}$ be a sequence of graphs such that $G_n \in \mathcal{H}(K)$ and $\delta(G_n) \to \infty$ as $n \to \infty$. We write $G \coloneqq G_n$.

    Let $\varepsilon >0$ and let
  \[
    p > \frac{1}{2}-\left(\frac{1}{2}-\eps\right) \sqrt{\frac{\log \dmax}{\dmax}}.
  \]
  Since $d \mapsto \frac{1}{2}-\left(\frac{1}{2}-\eps\right) \sqrt{(\log d)/d}$ is an increasing function in $d$, by \Cref{l:super-exponential} there exists $\beta >0$ such that for every $x \in V(G)$
  \[
\Pr*{x \notin A_{11}} \leq \exp(-\beta d(x)^2) \leq \exp(- \beta \dmin^2).
  \]
  Hence, by \Cref{c:bound} we have $|V(G)| \leq \exp(K\delta(G))$ and so
  \begin{align}
   1 - \Phi(p,G)  = \Pr{{\bf A}_p \text{ does not percolate on } G}\
   &\leq \Pr{ \exists x \in V(G): x \notin A_{11}} \nonumber\\&\leq |V(G)| \cdot \exp(- \beta \delta(G)^2) =o(1),\label{e:1statement:unionbound}
 \end{align}
 concluding the proof.
\end{proof}

\begin{remark}
    As mentioned before, \Cref{l:constant,l:dimred,l:super-exponential} do not use \Cref{c:bound}. The union bound in the last step of~\eqref{e:1statement:unionbound} is the only use of \Cref{c:bound} in the proof above. Therefore, the $1$-statement also holds for graphs satisfying \Cref{c:regular,c:backwards,c:localconnection,c:projection,c:partitioncond} and $|V(G)| = \exp(o(\delta(G)^2))$.
\end{remark}

\section{Proof of the 0-statement of Theorem \ref{t:mainThm}}\label{s:0-statement}

In this section we will prove the 0-statement of \Cref{t:mainThm}: For any $\varepsilon >0$, if
  \[
    p < \frac{1}{2}-\left(\frac{1}{2}+\eps\right) \sqrt{\frac{\log \dmin}{\dmin}},
  \]
then $$ \Phi(p,G):= \Pr{{\bf A}_p \text{ percolates on } G} \to 0.$$
Along the proof we will need some auxiliary lemmas, which are proved in \Cref{sec:dominating}. In fact, instead of directly analysing the majority bootstrap percolation process, we will analyse a generalised process which dominates the original majority bootstrap percolation process. This new process introduced in~\cite{BaBoMo2009} is called the \emph{$\Boot_k(\gamma)$ process}:  
Given $k \in \mathbb{N}$ and some function $\gamma \colon V(G) \to \mathbb{R}^+$, we recursively define $\hat{A}_0 := A_0$ and for each $\ell \in \mathbb{N} \cup \{0\}$

  \[
  \hat{A}_{\ell+1} := \hat{A}_\ell \cup \left\{ x \in V(G) \;:\; \left|N(x) \cap \hat{A}_\ell \right| \geq \frac{d(x)}{2} - \max\{0, k - \ell\}\cdot \gamma(x) \right\}.
  \]
In other words, the initial infection set is the same as in the majority bootstrap percolation process, but the infection spreads more easily in the first $k$ rounds. More precisely, a vertex $x$ is infected in the first round if it has $d(x)/2 - k\cdot\gamma(x)$ infected neighbours, and this requirement is gradually strengthened over the first $k$ rounds.
After the $k$-th round, the process evolves exactly as the majority bootstrap percolation process would do.

In particular, given $A_0$, we note that $$A_i\subseteq  \hat{A}_i \quad \text{ for all}\quad i \in \mathbb{N}.$$ Crucially, however, if a vertex $x$ becomes infected in round $\ell+1 \leq k$ of the $\Boot_k(\gamma)$ process, then at least $\gamma(x)$ of its neighbours must have become infected in round $\ell$.
This simplifies the task of showing a vertex does \emph{not} become infected. 

For our application, we fix $k = 2$ and for  each $x \in V(G)$ we let
\[ \gamma(x) := \sqrt{\frac{d(x)}{\vartheta(d(x))}}, \]
where $\vartheta(d) = \sqrt{\log d}$.
Our goal is to show that if $p$ is {small enough}, i.e, $p \leq \frac{1}{2}-c \sqrt{\frac{\log \dmin}{\dmin}}$ for $c>1/2$, then whp $$\hat{A}_2 = \hat{A}_3 \neq V(G).$$
 In other words, whp the process $\Boot_2(\gamma)$ {\em stabilises after two rounds}. Therefore, 
 $$\bigcup_{i=0}^\infty A_i \subseteq \bigcup_{i=0}^\infty \hat{A}_i = \hat{A}_2 \neq V(G)$$ and hence the majority bootstrap percolation process does not percolate.

Note that our choice of the parameter $\gamma(\cdot)$ simplifies the argument presented in \cite{BaBoMo2009} and allows us to study the first three instead of four rounds of the process, while obviating the need for some of the counting arguments given in \cite{BaBoMo2009}.
The choice of this parameter has to be such that $\gamma(x)$ is asymptotically smaller than $d(x) \sigma(x) = \sqrt{d(x) \log d(x)}$.
Our choice of $\gamma(x)$ to be very close to this bound allows us to simplify the argument, at the cost of giving a weaker bound on the width of the critical window.

We start by showing that it is very likely that the $\Boot_2(\gamma)$ process stabilises by the second round, by bounding the probability that a vertex is infected in the third round.
Recall that, as defined in \Cref{s:1-statement},
$$\sigma(x) := \sqrt{\frac{\log d(x)}{d(x)}}$$ for each vertex $x$.

\begin{restatable}{lemma}{Athree}\label{l:inA3-A2}
  Let $x \in V(G)$, $c > \frac{1}{2}$ and $p = \frac{1}{2}-c \sigma(x)$.
  If there is a $K \in \mathbb{N}$ such that $G \in \mathcal{H}(K)$, then there exists a $\beta >0$ (independent of $x$) such that
  \[ \Pr{x \in \hat{A}_3 \setminus \hat{A}_2} \leq \exp\left(-\beta \gamma(x)^2 \log d(x)\right).\]
\end{restatable}
We note that the proof of \Cref{l:inA3-A2} uses \Cref{c:regular,c:backwards,c:localconnection} (see \Cref{sec:dominating}). 

We will also need to show that it is unlikely that the $\Boot_2(\gamma)$ process fully percolates by the second round. Since it is very likely that around half of the vertices are initially infected, and we can quite precisely bound the probability that a vertex is infected in the first round, it will be sufficient to bound the probability that a vertex is infected in the second round.

\begin{restatable}{lemma}{Atwo}\label{l:inA2-A1}
Let $x \in V(G)$, $c > \frac{1}{2}$ and $p = \frac{1}{2}-c \sigma(x)$.
  If there is a $K \in \mathbb{N}$ such that $G \in \mathcal{H}(K)$, then
  \[ \Pr{x \in \hat{A}_2 \setminus \hat{A}_1} \leq \exp\left(-\sqrt{d(x)}\right).\]  
\end{restatable}
We note that the proof of \Cref{l:inA2-A1} needs \Cref{c:regular,c:backwards,c:localconnection} (see \Cref{sec:dominating}).

Finally, we observe that  \Cref{l:inA3-A2,l:inA2-A1} together with \Cref{c:bound} imply that whp the $\Boot_2(\gamma)$-process stabilises after the second round, without fully percolating. We are now ready to prove the $0$-statement of our main theorem.
\begin{proof}[Proof of the 0-statement of \Cref{t:mainThm}]
    Let $K \in \mathbb{N}$ and let $(G_n)_{n \in \mathbb{N}}$ be a sequence of graphs such that $G_n \in \mathcal{H}(K)$ and $\delta(G_n) \to \infty$ as $n \to \infty$.
    We write $G:= G_n$.

  Let $\varepsilon >0$ and let  \[
    p \leq \frac{1}{2}-\left(\frac{1}{2}+\eps\right) \sqrt{\frac{\log \dmin}{\dmin}}.
  \]
  Again, since the function $d \mapsto \frac{1}{2}-\left(\frac{1}{2}+\eps\right) \sqrt{(\log d)/d}$ is increasing in $d$, it follows from \Cref{l:inA2-A1} that for every $x \in V(G)$
  \[
  \Pr[\big]{x \in \hat{A}_2 \setminus \hat{A}_1} = o(1).
  \]
  Furthermore, since $\gamma(x) = o(\sigma(x) d(x))$, it is a simple consequence of Chernoff's inequality (\Cref{l:Chernoff}) that for every $x \in V(G)$
  \[
 \Pr[\big]{x \in \hat{A}_1 \setminus \hat{A}_0} \leq \Pr[\Big]{\Bin\bigl(d(x), p\bigr) \geq \frac{d(x)}{2}-2\gamma(x)} = o(1).
  \]
Hence, by Markov's inequality whp $\big|\hat{A}_2 \setminus \hat{A}_0 \big| = o\bigl(|V(G)|\bigr)$.
It follows from Chernoff's inequality (\Cref{l:Chernoff}) that whp $\big|\hat{A}_0\big|=\big|A_0\big| \leq \frac{3}{4} |V(G)|$, and hence whp
  \[
  \big|\hat{A}_2\big| = \big|\hat{A}_2 \setminus \hat{A}_0 \big| + \big|\hat{A}_0\big| < |V(G)|.
  \]
  On the other hand, by \Cref{l:inA3-A2} there exists $\beta > 0$ such that for every $x \in V(G)$
  \[
  \Pr[\big]{x \in \hat{A}_3 \setminus \hat{A}_2} = \exp\left( - \beta\gamma(x)^2 \log d(x) \right) = \exp\bigl( - \omega(\delta(G))\bigr).
  \]
  By \Cref{c:bound}, we have $|V(G)| \leq \exp(K \delta(G))$, and using the union bound we get that whp $\hat{A}_3 =  \hat{A}_2$. It follows that whp $A_i \subseteq \hat{A}_i = \hat{A}_2 \neq  V(G)$ for all $i \geq 2$. Therefore
  \[
  \Phi(p,G) := \Pr{{\bf A}_p \text{ percolates on } G}= \Pr*{\bigcup_{i=0}^\infty A_i  = V(G)} = o(1),
  \]
completing the proof.
\end{proof}

\section{Eleven rounds suffice}\label{sec:11suffice}

In this section we prove the auxiliary lemmas (\Cref{l:constant,l:dimred,l:super-exponential}) needed for the proof of the $1$-statement of \Cref{t:mainThm} in \Cref{s:1-statement}.

Throughout this section we let $x\in V(G)$, $c < \frac{1}{2}$, $p = \frac{1}{2} - c\sigma(x)$, where $\sigma(x) = \sqrt{\frac{\log d(x)}{d(x)}}$, and we fix a $K \in \mathbb{N}$ such that $G \in \mathcal{H}(K)$. We start the section with a simple corollary of the Central Limit Theorem (\Cref{l:CLT}).

\begin{lemma}\label{l:CLT-corollary}
  Let $d' = \Theta(d(x))$ and $C \in \mathbb{R}$ be a constant. Then
  \[ \Pr*{\Bin(d', p) \geq \frac{d'}{2} + C} = (1+o(1))\Pr*{\Bin(d', p) \geq \frac{d'}{2}}. \]
\end{lemma}
\begin{proof}
   Observe that since $(d')^{-1/2} \ll \sigma(x) \ll (d')^{-1/3}$, the function
    \[ f_L(d') \coloneqq \frac{cd' \sigma(x) + L}{\sqrt{d'p(1-p)}} \]
    satisfies $1 \ll f_L \ll (d')^{1/6}$ for any constant $L$. 
    Therefore, by \Cref{l:CLT}, it suffices to show that
    \begin{equation}\label{e:CLT-corollary-pt1}
    \Pr{N(0,1) \geq f_C(d')} = (1+o(1)) \Pr{N(0,1) \geq f_0(d')}.
    \end{equation}
    Since $|f_C - f_0| = O((d')^{-1/2})$, we may estimate
    \[
    \frac{1}{f_C\cdot\sqrt{2\pi}} \cdot \exp\left(-\frac{f_C^2}{2}\right) = \frac{1+o(1)}{f_0 \cdot \sqrt{2\pi}} \cdot \exp\biggl(-\frac{f_0^2}{2} + O\bigl((f_C-f_0)\cdot f_0\bigr)\biggr) = \frac{1+o(1)}{f_0\cdot \sqrt{2\pi}} \cdot \exp\left(-\frac{f_0^2}{2}\right). 
    \]
    By the second part of~\Cref{l:CLT}, this implies~\eqref{e:CLT-corollary-pt1}, hence proving the lemma.
\end{proof}

We proceed to prove the aforementioned auxiliary lemmas.

\subsection{Proof of Lemma \ref{l:constant}}

As in \Cref{l:constant}, we let $m \in \mathbb{N} \cup \{0\}$ and aim to prove 
  \[
    \Pr*{x \in A_2(m)} \geq \frac{3}{4} + o(1).
  \]

Let $X_0:=N(x) \cap A_0$ be the set of neighbours of $x$ which are initially infected and note that 
$$|X_0| \sim \Bin(d(x),p).$$ 
Next we let \[X_1:=N(x) \cap (A_1(m)\setminus A_0)\] be the set of the neighbours of $x$ which become infected in the first round.  We note that $x \in A_2(m)$ if and only if either $x \in A_0$ or $|X_0|+|X_1| \geq \frac{d(x)}{2}+m$, and so we have
  \begin{equation}\label{e:splitprob}
    \Pr*{ x \in A_2(m) } = \Pr*{x \in A_0 }+  \Pr*{x \notin A_0} \cdot \Pr[\Big]{|X_0|+|X_1| \geq \frac{d(x)}{2}+m \given x \notin A_0 }.
  \end{equation}

  In order to calculate the second term in \eqref{e:splitprob} we consider the event $\mathcal{E}_0$ that $|X_0| \geq \frac{d(x)}{2}+m-\ell$ and the event $\mathcal{E}_1$ that $|X_1| \geq \ell$ for a judicious choice of $m \leq \ell \leq \frac{d(x)}{2}+m$ which will be given later.
  By \eqref{e:splitprob}, we have
  \begin{align}
    \Pr*{x \in A_2(m) } & \geq \Pr*{x \in A_0 }+ \Pr*{x \notin A_0} \cdot \Pr*{\mathcal{E}_0 \wedge \mathcal{E}_1 \given x \notin A_0} \nonumber   \\
    & = p + (1-p)\cdot\Pr*{ \mathcal{E}_0 \given x\not\in A_0} \cdot\Pr*{\mathcal{E}_1 \given \mathcal{E}_0 \wedge \left(x \notin A_0\right)} \nonumber \\
    & =p + (1-p)\cdot \Pr*{ \mathcal{E}_0} \cdot \Pr*{\mathcal{E}_1 \given \mathcal{E}_0 \wedge \left(x \notin A_0\right)} \label{e:splitprobcond},
  \end{align}
  where the last equality is because $\mathcal{E}_0$ is independent of the event that $x \notin A_0$.

For ease of notation, let us write $\mathbb{P}^*$ for the probability distribution conditioned on the events $\{x \notin A_0\}$ and $\mathcal{E}_0$. Note that, for $y \in N(x)$, we may use
the independence of the events $\{y \notin A_0\}$ and $\{x \notin A_0\}$ as well as \Cref{l:Binvariant} on correlations to deduce that
\begin{align}\label{e:ynotingivenR}
    \Prs*{y \notin A_0} = \Pr{y \notin A_0 \mid  (x \notin A_0) \wedge \mathcal{E}_0} = \Pr{y \notin A_0 \given \mathcal{E}_0} = (1+o(1))(1-p).
\end{align}

We begin with estimating the  conditional expectation of $|X_1|$.
  \begin{claim}\label{c:exp}
    \[\Exs[\big]{|X_1| }  = \Ex[\big]{|X_1|  \mid (x \notin A_0) \wedge \mathcal{E}_0} = \Omega\left(\frac{d(x)^{1-2c^2}}{\sqrt{\log d(x)}}\right).
    \]
  \end{claim}
\begin{proof}[Proof of \Cref{c:exp}]
  Let us suppose that $x \notin A_0$ and let $y \in N(x)$. We have
  \begin{align}\label{e:splitprobyins}
    \Prs[\Big]{y \in X_1} = \Prs{y \notin A_0}\cdot \Prs[\Big]{|N(y) \cap A_0| \geq \frac{d(y)}{2}+m \given y \notin A_0}.
  \end{align}

  To bound the second term in \eqref{e:splitprobyins}, we note that $\big|N(y) \cap \big(\{x\} \cup N(x)\big)\big| \leq K$ by \Cref{c:backwards}.
  Furthermore, we have $d(y) \geq d(x) - K$ since $G$ is $K$-locally almost regular by \Cref{c:regular}.
  Hence, conditioned on the events $\{x, y \notin A_0\}$ and $\mathcal{E}_0$, $|N(y) \cap A_0|$ stochastically dominates a $\Bin(d', p)$ random variable for $d' := d(x)-2K$. Recalling that $p = \frac{1}{2} -c\sigma(x) = \frac{1}{2} - c \sqrt{\frac{\log d(x)}{d(x)}}$, we have
  \begin{equation}\label{e:infectedneighbourconditional}
    \Prs*{|N(y) \cap A_0| \geq \frac{d(y)}{2}+m \given y \not\in A_0} \geq \Pr*{\Bin\left(d', p\right)\geq \frac{d'}{2}+m + \frac{3}{2}K}.
  \end{equation}
From \eqref{e:splitprobyins}, \eqref{e:ynotingivenR}, \eqref{e:infectedneighbourconditional} and \Cref{l:CLT-corollary}, we obtain
\begin{align}\label{e:splitprobyins2}
    \Prs[\big]{y \in X_1} \geq (1+o(1))(1-p) \cdot \Pr*{\Bin\left(d', p\right)\geq \frac{d'}{2}}.
  \end{align}

We will bound the second term on the right-hand side of \eqref{e:splitprobyins2} by applying \Cref{l:CLT} for $\Bin(d', p)$. To do that, we need to scale the binomial random variable accordingly. Note that
  \[
    d' \left(\frac{1}{2} - p\right) = c d' \sigma(x) =  \left(1 + O\left(\frac{1}{d(x)}\right)\right) c \sqrt{d(x) \log d(x)}
  \]
  and
  \[
    d' p(1-p)=d' \left(\frac{1}{4}-c^2\sigma(x)^2\right)=\left(1+O\left(\frac{\log d(x)}{d(x)}\right)\right) \frac{d(x)}{4}.
  \]
  We may apply \Cref{l:CLT} to bound the right-hand side of \eqref{e:splitprobyins2} and get
  \begin{equation}\label{e:estimatingBinomial}
       \Pr*{\Bin\left(d', p\right)\geq \frac{d'}{2}}=(1+o(1)) \Pr{N(0,1)\geq f(d')}
  \end{equation}
  for
  \[
    f(d') := \frac{d'(1/2 - p)}{\sqrt{d' p(1-p)}} = \left(1+O\left(\frac{\log d(x)}{d(x)}\right)\right) 2c \sqrt{\log d(x)}.
  \]
  For this value of $f$, we may estimate
  \begin{align}
       \Pr{N(0,1)\geq f(d')} &=\frac{1+o(1)}{f(d') \sqrt{2 \pi}} \exp\left(-\frac{1}{2}f(d')^2\right)          \nonumber\\
       &=\frac{1+o(1)}{2c\sqrt{2\pi\log d(x)}} \exp\bigl(-2 c^2 \log d(x) + O(1)\bigr) \nonumber\\
       &=\Omega\left(\frac{d(x)^{-2c^2}}{\sqrt{\log d(x)}}\right).\label{e:estimatingNormal}
  \end{align}
  Hence, by \eqref{e:splitprobyins2}, \eqref{e:estimatingBinomial} and \eqref{e:estimatingNormal}, we obtain
  \begin{equation}\label{e:expectation}
      \Exs*{|X_1| } =   d(x) \cdot  \Prs[\Big]{y \in X_1} = \Omega\left(\frac{d(x)^{1-2c^2}}{\sqrt{\log d(x)}}\right),
  \end{equation}
  as desired. 
\end{proof}

We now proceed towards bounding the conditional variance of $|X_1|$. Let $M$ be the set of pairs of distinct vertices $y, z \in N(x)$ that have \emph{at most} two common neighbours (including $x$). We continue with the following claim, which states the events $\{y \in X_1\}$ and $\{z \in X_1\}$ for such pairs are not closely correlated.

  \begin{claim} \label{c:Var1N}
    If $(y,z) \in M$, then
    \[
      \Prs*{(y \in X_1) \wedge (z \in X_1)} - \Prs{y \in X_1} \cdot \Prs{z \in X_1} = o\left(\Prs{y \in X_1} \cdot \Prs{z \in X_1}\right).
    \]
  \end{claim}

  \begin{proof}
    Let $d' = d(x) - 2K$, $d'_0 = d(x)/2 + m - 4K$ and $d'_1 = d(x)/2 + m + 4K$. 
    We claim that for $w \in \{y, z\}$,
    \begin{equation}\label{e:Var1N:dom1}
    \Pr[\big]{\Bin(d', p) \geq d'_0} \geq \Prs*{w \in X_1 \given w \notin A_0} \geq \Pr[\big]{\Bin(d', p) \geq d'_1}.
    \end{equation}
    Indeed, assume $w = y$ without loss of generality, and observe that
    \begin{align*}
    \Prs*{y \in X_1 \given y \notin A_0} &\geq \Prs*{y \in X_1 \given (y \notin A_0) \wedge (N(y) \cap N(x) \cap A_0 = \emptyset)} \\
    &= \Pr*{y \in X_1 \given (x, y \notin A_0) \wedge (N(y) \cap N(x) \cap A_0 = \emptyset)}\\
    &\geq \Pr*{\Bin(d(y) - K, p) \geq d(y)/2 + m}
    \end{align*}
    where the first step follows by stochastic domination. Since $d(y) = d \pm K$ by \Cref{c:regular},  this proves the last inequality in~\eqref{e:Var1N:dom1}. A similar argument, conditioning on $N(y) \cap N(x) \subseteq A_0$ instead and using that $\Pr{\Bin(d', p) \geq d_0'} \geq \Pr{\Bin(d(y), p) \geq d_0'+3K}$, proves the first inequality.
    
    Let $F = \{w \in V(G) \setminus \{x,y, z\} : |N(w) \cap \{x, y, z\}| \geq 2\}$, and observe that $|F| \leq 2K$ by \Cref{c:backwards} and our assumption that $y$ and $z$ have at most one common neighbour besides $x$. Since $F$ contains $N(y) \cap N(z)$, if we condition on $(F \cup \{x,y, z\}) \cap A_0$ then the events $\{y \in X_1\}$ and $\{z \in X_1\}$ are conditionally independent and binomially distributed and so by a similar argument, conditioning on either $F = \emptyset$ or $F \subseteq A_0$ allows us to deduce that
    \begin{equation}\label{e:Var1N:dom2}
        \Pr[\big]{\Bin(d', p) \geq d'_0}^2 \geq \Prs*{(y \in X_1) \wedge (z \in X_1) \given y, z \notin A_0} \geq \Pr[\big]{\Bin(d', p) \geq d'_1}^2. 
    \end{equation} 
    By~\eqref{e:Var1N:dom1},~\eqref{e:Var1N:dom2} and~\Cref{l:CLT-corollary}, we obtain
    \begin{equation}\label{e:Var1N:almost}
       \Prs*{(y \in X_1) \wedge (z \in X_1) \given y, z \notin A_0} = (1+o(1)) \Prs*{y \in X_1 \given y \notin A_0}\cdot \Prs*{z \in X_1 \given z \notin A_0}.
    \end{equation}
    Therefore, it only remains to show that 
    \begin{equation}\label{e:Var1N:bayes}
    \Prs{y, z \notin A_0} = (1+o(1)) \Prs{y \notin A_0} \cdot \Prs{z \notin A_0},
    \end{equation}
    since multiplying~\eqref{e:Var1N:almost} and~\eqref{e:Var1N:bayes} finishes the proof of the lemma. 
    To do so, let $\mathbb{P}'$ be the probability measure obtained by conditioning on $\{x \notin A_0\}$, and observe that, by Bayes' theorem,
    \begin{equation}\label{e:Var1N:bayes0}
    \Prs{y, z \notin A_0} = \Prp{y, z \notin A_0 \given \mathcal{E}_0} = \frac{\Prp{\mathcal{E}_0 \given y, z \notin A_0}}{\Prp{\mathcal{E}_0}} \cdot \Prp{y, z \notin A_0}.
    \end{equation}
    The fraction is $1+o(1)$ by \Cref{l:Binvariant}. Since $\Prp{y, z \notin A_0} = \Pr{y, z \notin A_0} = (1-p)^2$ and $\Prs{y \notin A_0}\Prs{z \notin A_0} = (1+o(1))(1-p)^2$ by \eqref{e:ynotingivenR}, \eqref{e:Var1N:bayes0} implies \eqref{e:Var1N:bayes} and proves the lemma.
  \end{proof}
  
  As a direct consequence of \Cref{c:Var1N} we have
\begin{align}
 \sum_{(y, z) \in M} \Prs*{(y \in X_1) \wedge (z \in X_1)} - \Prs{y \in X_1} \cdot \Prs{z \in X_1} =  o\left(\Exs{|X_1|}^2\right).\label{eq:varMpart}
\end{align}

This allows us to obtain our bound on the conditional variance of $|X_1|$.
  \begin{claim}\label{c:var}
    \[\mathbb{V}^*\left[ |X_1|  \right]= o\left(\Exs{|X_1|}^2\right). 
\]
  \end{claim}
 \begin{figure}
    \centering
    \begin{tikzpicture}[scale=1.2, thick, every node/.style={circle}]
      \node (V) at (-1,0) [circle,draw, fill, scale=0.5]  {};
      \node[] at (-1,0.4) {$x$};
      \begin{scope}
         \draw[clip, rotate=90] (0, -3.8) ellipse (3cm and 0.8cm);
         \fill[gray] (3,2) -- (3.6,4) -- (4,4) -- (4.5,2) -- cycle;
      \end{scope}
      \begin{scope}
         \draw[clip, rotate=90] (0, -2) ellipse (2.8cm and 0.8cm);
         \fill[gray] (1,2.4) -- (3,2.4) -- (3,3.5) -- (1,3.5) -- cycle;
      \end{scope}

      \draw[rotate=90] (0, -2) ellipse (2.8cm and 0.8cm);
      \node[] at (1.9,3.1) {$N(x)$};
      \path[draw=black] (-1,0) -- (1.8,2.75);
      \path[draw=black] (-1,0) -- (1.8,-2.75);
      \draw[rotate=90] (0, -3.8) ellipse (3cm and 0.8cm);
      \node[] at (3.8,3.3) {$S(x, 2)$};

      \draw[rotate=90, fill=blue, fill opacity=0.4] (1.2, -2) ellipse (1cm and 0.3cm);
      \node[] at (1.53,0.7) {$X_0$};
      \draw[rotate=90] (-1.3, -2) ellipse (1cm and 0.35cm);
      \fill[pattern=north east lines, rotate=90] (-1.3, -2) ellipse (1cm and 0.35cm);
      \node[] at (1.55,-0.3) {$X_1$};

      \node (W1) at (2,-0.7) [circle,draw, fill=red, scale=0.5] {};
      \node (W2) at (2,-1.1) [circle,draw, fill=red, scale=0.5] {};
      \node (W2) at (2,-1.7) [circle,draw, fill=yellow, scale=0.5] {};
      \node (W2) at (2,-2) [circle,draw, fill=yellow, scale=0.5] {};

      \node (W2) at (3.8,2.6) [circle,draw, fill, scale=0.5] {};
      \node (W2) at (3.8,0.6) [circle,draw, fill, scale=0.5] {};
      \node (W2) at (3.8,-1.6) [circle,draw, fill, scale=0.5] {};
      \path[draw=black] (2,-0.7) -- (3.8,2.6);
      \path[draw=black] (2,-0.7) -- (3.8,0.6);
      \path[draw=black] (2,-1.1) -- (3.8,0.6);
      \path[draw=black] (2,-2) -- (3.8,-1.6);
      \path[draw=black] (2,-1.7) -- (3.8,-1.6);
      \path[draw=black] (2,-1.1) -- (3.8,2.6);
  \end{tikzpicture}
  \caption{The set $X_0 =N(x) \cap A_0$ is depicted in blue, and the set $X_1=N(x) \cap (A_1 \setminus A_0)$ is depicted patterned. The red pair of vertices has two common neighbours in $S(x, 2)$ and is thus in $M'$, while the yellow pair of vertices has just one common neighbour in $S(x, 2)$.}
  \label{f:Variance}
  \end{figure}
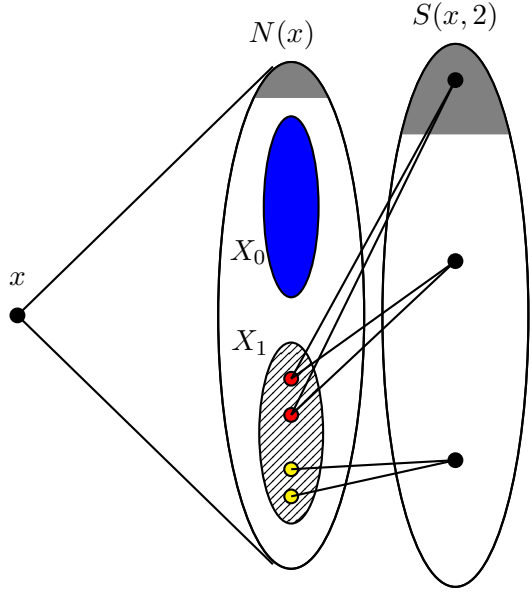
  \begin{proof}[Proof of \Cref{c:var}]
  Let $M'$ be the set of pairs of distinct vertices in $N(x)$ that have strictly more than two common neighbours (see Figure \ref{f:Variance}), which we will call {\em bad pairs}.
  We claim that $|M'| = O(d(x))$.
  Indeed, $|\binom{N(x)}{2} \setminus \binom{S_0(x,1)}{2}| = O(d(x))$
  by \Cref{c:localconnection}\ref{i:small}, so there are only $O(d(x))$ bad pairs not contained in $S_0(x,1)$.
      Also due to \Cref{c:localconnection}\ref{i:small}, we have that $N(x) \cup (D(x) \cap S(x,2))$ has size $O(d(x))$, and due to \Cref{c:backwards} each of its elements can be a common neighbour of $O(1)$ many pairs of neighbours of $x$, contributing  $O(d(x))$ many bad pairs in total.
  Since by \Cref{c:localconnection}\ref{i:cherry} any two vertices in $S_0(x, 1)$ have at most one common neighbour in $S_0(x, 2)$, there are no further bad pairs to consider.
  Summing up, we have 
   \begin{align}
   |M'| =O(d(x)).\label{eq:var:Mcom}
   \end{align}
   
From \eqref{eq:varMpart} and \eqref{eq:var:Mcom} it follows that
  \begin{align*}
    \mathbb{V}^*\left[ |X_1|  \right] & = \sum_{y, z \in N(x)} \Prs*{(y \in X_1) \wedge (z \in X_1)} - \Prs{y \in X_1} \cdot \Prs{z \in X_1} \notag \\
    & \leq \sum_{(y, z) \in M} \Prs*{(y \in X_1) \wedge (z \in X_1)} - \Prs{y \in X_1} \cdot \Prs{z \in X_1} \notag\\
    & \qquad + \Exs{|X_1|} + |M'| \notag \\
    & \leq o\left(\Exs{|X_1|}^2\right) + O(d(x)) = o\left(\Exs{|X_1|}^2\right), 
  \end{align*}
  finishing the claim. 
  \end{proof}

\begin{proof}[Proof of \Cref{l:constant}]  
Recall that $|X_0| \sim \Bin(d(x),p)$, where $p=\frac{1}{2}-c \sigma(x) = \frac{1}{2}-c \sqrt{\frac{\log d(x)}{d(x)}} $ for $c < 1/2$. Taking $\ell: = \Exs*{|X_1|}/2$ we have $\ell = \Omega\bigl(d(x)^{1-2c^2}/\sqrt{\log d(x)}\bigr)$ by \Cref{c:exp}, and thus
   \begin{equation}\label{e:expE0}
    \Ex{|X_0|} = \left(\frac{1}{2}-c \sigma(x)\right) d(x) = \frac{d(x)}{2} - c \sqrt{d(x)\log d(x)} > \frac{d(x)}{2} + m - \ell.
  \end{equation}
By \Cref{p:MedianBin} we then have 
  \begin{equation}\label{e:probE0}
  \Pr{\mathcal{E}_0} = \Pr*{|X_0|\geq \frac{d(x)}{2}+m-\ell} \geq \frac{1}{2}.
  \end{equation}
  Furthermore, by \Cref{c:exp,c:var} it follows from Chebyshev's inequality that
  \begin{equation} \label{eq:probell}
    \Prs*{\mathcal{E}_1} = 1 - \Prs*{|X_1| < \ell} = 1- \Prs*{|X_1| < \frac{ \Exs*{|X_1|}}{2}}=1-o(1).
  \end{equation}
   Recalling from \eqref{e:splitprobcond} that $\Pr{x \in A_2(m)} \geq p +(1-p)\cdot \Pr{\mathcal{E}_0}\cdot \Prs{\mathcal{E}_1}$ and using \eqref{e:probE0} and \eqref{eq:probell},
  \begin{align*}
    \Pr*{x \in A_2(m) } & \geq  p + (1-o(1)) \frac{1-p}{2} \geq \frac{3}{4} +o(1),
  \end{align*}
finishing the proof of the lemma.
\end{proof}

\begin{remark}\label{r:finer1statement}
Let $\lambda \in \mathbb{R}$ satisfy $\lambda > 1/2$, and assume now that $c$ may depend on $d(x)$. The inequality in \eqref{e:expE0} holds as long as  
\[ c < \frac{1}{2} - \frac{\lambda \log \log d(x)}{\log d(x)}. \]
For the formal argument, see \Cref{s:appendix}. Using this, the $1$-statement of \Cref{t:mainThm} indeed holds for every
\[ p \geq \frac{1}{2} - \frac{1}{2} \sqrt{\frac{\log \Delta(G)}{\Delta(G)}} + \frac{\lambda \log \log \Delta(G)}{\sqrt{\Delta(G) \log \Delta(G)}}, \] recovering the bound obtained in \cite[Theorem 2.1]{BaBoMo2009} (see \Cref{t:BBM}) if $G \in \mathcal{H}(K)$ is $n$-regular.
\end{remark}

\subsection{Proof of Lemma \ref{l:dimred}}

In this section we will prove \Cref{l:dimred} about the existence of a constant $\beta >0$ (independent of $x$) such that
  \[
    \Pr[\big]{ x \notin A_5 } \leq \exp\bigl( - \beta d(x)\bigr).
  \]
 Throughout this section we let $0 < \alpha_1 \ll \alpha_2 \ll \alpha_3 \ll K^{-10}$  be sufficiently small constants and let $d=d(x)$. 

Suppose that $x \notin A_5$.  Then, fewer than $\frac{d}{2}$ of the vertices in $N(x)$ are infected by the fourth round.
  In particular, there is some subset $T' \subseteq N(x)$ with $|T'|=\frac{d}{2}$ such that $T' \cap A_4 =\emptyset$.
  By a similar argument, using \Cref{c:regular}, each vertex $y \in T'$ has fewer than $\frac{d(y)}{2}\leq \frac{d+K}{2}$ neighbours in $A_3$.
  Since $T' \subseteq N(x) = S(x,1)$, it follows that
  \begin{equation}\label{e:review}
    \phi(T') := e\left(T',A_3 \cap S(x,2) \right) < \frac{(d+K)d}{4}.
  \end{equation}
 
 We will show that the probability that there is a subset $T' \subseteq N(x) \setminus A_4$ of size $|T'| = \frac{d}{2}$ such that \eqref{e:review} holds is at most $\exp\bigl( - \beta d\bigr)$.

  Firstly, let $\mathcal{E}$ be the event that $|N(x) \cap A_0| \geq \left(\frac{1}{2} - \alpha_1 \right)d$.
  Since $|N(x) \cap A_0| \sim \Bin(d,p)$ and $p = \frac{1}{2} - o(1)$, by Chernoff's inequality (\Cref{l:Chernoff}) it holds that
  \begin{equation}\label{e:conditionR}
    \Pr{\mathcal{E}} \geq 1 - \exp\bigl(-\alpha_1^2 d\bigr).
  \end{equation}

  In what follows, we will condition on $N(x) \cap A_0$ and assume that $\mathcal{E}$ holds, writing $\mathbb{P}_1$ for the conditional probability space to avoid having to make this conditioning explicit everywhere. Whilst we may at points consider the probability of events under $\mathbb{P}_1$ which depend on $N(x) \cap A_0$, we will always bound these probabilities by the probability of events which only depend on the status of vertices in $V \setminus B(x,1)$, which are independent of $N(x) \cap A_0$.

  Hence, in order to prove \Cref{l:dimred}, it will be sufficient to show that
  \begin{equation}\label{e:boundmsum}
    \sum_{\substack{T' \subseteq N(x) \setminus A_0\\ |T'| = \frac{d}{2}}} \Prone*{\phi(T') < \frac{(d+K)d}{4}} \leq e^{-\alpha_2 d}.
  \end{equation}
  Note that, given a realisation of $N(x) \cap A_0$ such that $\mathcal{E}$ holds, the number of possible sets $T'\subseteq N(x) \setminus A_0$ with $|T'| = \frac{d}{2}$ is at most
  \begin{equation}\label{e:numberofS}
    \binom{\left(\frac{1}{2} + \alpha_1\right)d}{d/2} = \binom{\left(\frac{1}{2} + \alpha_1\right)d}{\alpha_1 d} \leq \left(\frac{e \left(\frac{1}{2}+\alpha_1\right)}{\alpha_1}\right)^{\alpha_1 d} \leq \left(\frac{e}{\alpha_1}\right)^{\alpha_1 d} \leq e^{\alpha_2 d},
  \end{equation}
  by our choice of $\alpha_1$ and $\alpha_2$. Therefore, the following claim readily implies \eqref{e:boundmsum}.
\begin{claim} \label{c:individualsum}
For each $T'\subseteq N(x)$ with $|T'|=\frac{d}{2}$,
  \begin{equation}\label{e:individualsum}
    \Prone*{\phi(T') < \frac{d^2}{4}  + c\sqrt{d^3 \log d}}\leq e^{- 2\alpha_2 d}.
  \end{equation}
  \end{claim}
 \begin{proof}[Proof of \Cref{c:individualsum}]
Let $T'\subseteq N(x)$ with $|T'|=\frac{d}{2}$ be fixed, and let $$T := N(T') \cap S(x,2)$$ (see \Cref{f:const-exp}). Note that $T \subseteq S(x,2)$, and so $T \cap A_0$ is independent of $N(x) \cap A_0$.

By \Cref{c:regular,c:backwards}, we have $e(T', T) = \frac{d^2}{2} + O(d)$ and $|T| \geq \frac{d^2}{4K} + O(d)$.
  In particular, since $p=\frac{1}{2} - c\sigma(x)$, the expected size of $e(T', T \cap A_0)$ is $\frac{d^2}{4} - \frac{c}{2}\sqrt{d^3 \log d} + o(d^{3/2})$ and so, by Hoeffding's inequality (the first inequality in \Cref{l:sumBin}) and the fact that every vertex of $T$ has at most $2K$ backwards neighbours by \Cref{c:backwards}, we obtain
  \begin{equation}\label{e:sizeofTA0}
    \Prone*{ e(T' , T \cap A_0) < \frac{d^2}{4} - c\sqrt{d^3 \log d} } < \exp\left(-\frac{c^2 d \log d}{2K}\right).
  \end{equation}

  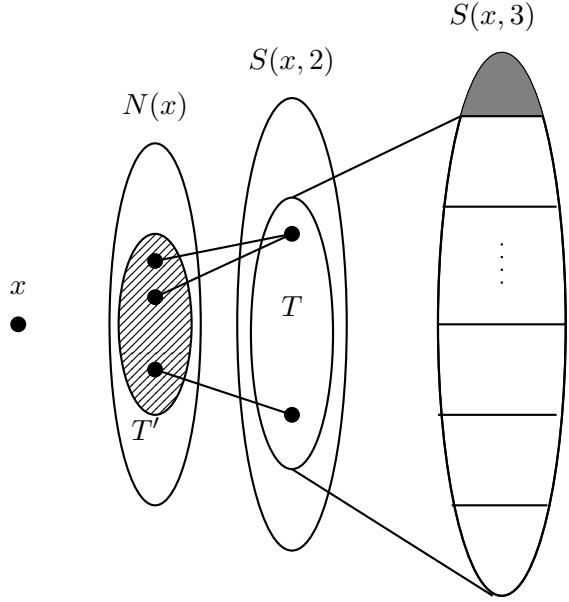
\begin{figure}
    \centering
    \begin{tikzpicture}[scale=1.2, thick, every node/.style={circle}]
      \node (V) at (0,0) [circle,draw, fill, scale=0.5]  {};
      \node[] at (0,0.4) {$x$};
      \draw[rotate=90] (0, -1.5) ellipse (2cm and 0.5cm);
      \node[] at (1.5,2.4) {$N(x)$};
      \draw[rotate=90] (0, -3) ellipse (2.5cm and 0.6cm);
      \node[] at (3,2.9) {$S(x, 2)$};
      \draw[rotate=90] (0, -5.3) ellipse (3cm and 0.7cm);
      \node[] at (5.2,3.4) {$S(x, 3)$};

      \begin{scope}
              \draw[clip, rotate=90] (0, -5.3) ellipse (3cm and 0.7cm);
              \fill[gray] (4.6,2.3) -- (4.6,4) -- (5.9,4) -- (5.9,2.3) -- cycle;
      \end{scope}
      \path[draw=black] (4.85,2.3) -- (5.75,2.3);

      \draw[rotate=90] (0, -1.5) ellipse (1cm and 0.4cm);
      \fill[pattern=north east lines, rotate=90] (0, -1.5) ellipse (1cm and 0.4cm);
      \node[] at (1.4,-1.15) {$T'$};
      \draw[rotate=90] (-0.1, -3) ellipse (1.5cm and 0.45cm);

      \node (W2) at (3,-1) [circle,draw, fill, scale=0.5] {};
      \node (W1) at (1.5,-0.5) [circle,draw, fill, scale=0.5] {};
      \path[draw=black] (3,-1) -- (1.5,-0.5);
      \node (U) at (3,1) [circle,draw, fill, scale=0.5] {};
      \node (U1) at (1.5,0.3) [circle,draw, fill, scale=0.5] {};
      \node (U2) at (1.5,0.7) [circle,draw, fill, scale=0.5] {};
      \path[draw=black] (3,1) -- (1.5,0.3);
      \path[draw=black] (3,1) -- (1.5,0.7);

      \path[draw=black] (4.6,0) -- (6,0);
      \path[draw=black] (4.6,-1) -- (5.9,-1);
      \path[draw=black] (4.75,-2) -- (5.8,-2);
      \path[draw=black] (4.65,1.3) -- (5.9,1.3);

      \draw [loosely dotted] (5.3,0.45) -- (5.3,0.95);

      \node[] at (3,0.2) {$T$};

      \path[draw=black] (3,1.4)-- (4.85,2.3);
      \path[draw=black] (3,-1.6)-- (5.2,-3);

  \end{tikzpicture}
  \caption{The set $T' \subseteq N(x) \setminus A_4$ is depicted patterned. We will estimate the density of certain types of edges between $T= N(T') \cap S(x, 2)$ and $S_0(x, 3)$, where $D(x)\cap S(x,3)$ is depicted in grey. Applying \Cref{l:partitiondist}, we obtain a partition $S_0(x, 3)$ into sets of vertices with pairwise distance at least $6$.}
  \label{f:const-exp}
  \end{figure}

  In particular, we can bound \eqref{e:individualsum} by showing that it is (exponentially in $d$) unlikely that very few vertices in $T$ are in $A_3 \setminus A_0$ as in the following claim, which will be proved after finishing the proof of \Cref{c:individualsum}.


  \begin{claim}\label{c:Tsteps13}
    $\displaystyle\Prone*{\left|T \cap \left(A_3 \setminus A_0\right)\right| \leq  2c\sqrt{d^3 \log d}} \leq \exp(-3\alpha_2 d).$
  \end{claim}
  
  To finish the proof, note that $e(T', T \cap (A_3 \setminus A_0)) \geq |T \cap (A_3 \setminus A_0)|$, since every vertex of $T$ has a neighbour in $T'$. It then follows from \eqref{e:sizeofTA0} and \Cref{c:Tsteps13} that
  \begin{align*}
    \Prone*{\phi(T') < \frac{d^2}{4}  + c\sqrt{d^3 \log d}} &\leq \Prone*{ e(T' , T \cap A_0) + \left|T \cap \left(A_3 \setminus A_0\right)\right| < \frac{d^2}{4} + c\sqrt{d^3 \log d} }\\
    &\leq  \exp\left(-\frac{c^2 d \log d}{2K}\right) + \exp(-3\alpha_2d) \leq \exp(-2\alpha_2d),
  \end{align*}
  establishing \Cref{c:individualsum}.
\end{proof}

  \begin{proof}[Proof of \Cref{c:Tsteps13}]
    Recall that $G$ is locally $K$-almost regular by \Cref{c:regular}, so every vertex in $T \subseteq S(x,2)$ has degree $d \pm 2K$.
    Since $G$ has $K$-bounded backwards expansion by \Cref{c:backwards},  every vertex in $S(x, 2)$ has at most $2K$ neighbours in $S(x,1) \cup S(x, 2)$.
    Furthermore, by \Cref{c:localconnection}\ref{i:sparse}, we have that every vertex $y \in S_0(x,2)$ has at most $2K$ neighbours in $D(x)$, and so it follows that every such $y$ has at least $d(y) - 4K \geq d - 6K$ neighbours in $S_0(x, 3)$.
    Moreover, $e(T \setminus S_0(x, 2), S_0(x,3)) = O(d^2)$ by Properties \ref{c:regular} and \ref{c:localconnection}\ref{i:small}. Therefore,
    \begin{equation}\label{e:edgesTtoS3}
      e\left(T \setminus A_0, S_0(x, 3)\right) = \left|T \setminus A_0\right| (d \pm 6K) + O(d^2) = \left|T \setminus A_0\right| d + O(d^2).
    \end{equation}
    Suppose now that the event in the claim statement, $\left|T \cap \left(A_3 \setminus A_0\right)\right| \leq 2c\sqrt{d^3 \log d}$, occurs.
    Every vertex $y \in T \setminus A_3$ has at most $\frac{d(y)}{2} \leq \frac{d}{2} + K$ neighbours in $S_0(x, 3) \cap A_2$ and hence
    \begin{align}
      e\left(T \setminus A_0, S_0(x, 3) \cap A_2\right) & \leq \left|T \setminus A_3\right|\left(\frac{d}{2}+K\right) + \left|T \cap \left(A_3 \setminus A_0\right)\right|(d + 2K) \nonumber \\
      & \leq \left|T \setminus A_0\right|\left(\frac{d}{2}+K\right) + 3c\sqrt{d^5 \log d}. \label{e:edgesTwo0toS3capA2}
    \end{align}
    Hence, it follows from \eqref{e:edgesTtoS3} and \eqref{e:edgesTwo0toS3capA2} that
    \begin{align}\label{e:edgeTwoA0toS3woA2}
      e\left(T \setminus A_0, S_0(x, 3) \setminus A_2\right) &= e\left(T \setminus A_0, S_0(x, 3)\right) - e\left(T \setminus A_0, S_0(x, 3) \cap A_2\right) \nonumber \\& \geq  \left|T \setminus A_0\right|\cdot\frac{d}{2} - 4c\sqrt{d^5 \log d}.
    \end{align}

    Therefore, the average density from $T \setminus A_0$ to $S_0(x, 3) \setminus A_2$ is at least around half the density from $T \setminus A_0$ to $S_0(x, 3)$.
    We will restrict ourselves to a subset $B$ of $S_0(x, 3)$ with a similar property such that distinct vertices in $B$ lie at distance at least six in $G$, in order to ensure independence of the events that the vertices of $B$ lie in $A_2$.

    By Chernoff's inequality (\Cref{l:Chernoff}), the event $|T \setminus A_0| \geq d^2/(8K)$ occurs with probability $1 - \exp\bigl(-\Omega(d \log d)\bigr)$.
    Using \Cref{l:partitiondist}, we can find a partition $\mathcal{P}$ of $S_0(x, 3)$ with $|\mathcal{P}| \leq 4K^3 d^2$ (see \Cref{f:const-exp}) such that, for each $P \in \mathcal{P}$, the family of events $\left\{y \in A_2 \;:\; y \in P\right\}$ is independent. Assuming $|T \setminus A_0| \geq d^2/(8K)$ holds, we claim that some $P \in \mathcal{P}$ must satisfy
    \begin{equation}\label{e:Tsteps13unlikely}
      e\left(T \setminus A_0, P \setminus A_2\right) \geq \max \left\{ \frac{5}{12} e\left(T \setminus A_0, P\right), \alpha_3 d \right\}.
    \end{equation}
    Indeed, assume every $P \in \mathcal{P}$ violates \eqref{e:Tsteps13unlikely}. Then let $\mathcal{P}_1$ be the family of sets $P \in \mathcal{P}$ such that $e\left(T \setminus A_0, P \setminus A_2\right) < \alpha_3 d$. Similarly, let $\mathcal{P}_2$ be the family of sets $P \in \mathcal{P}$ such that $e\left(T \setminus A_0, P \setminus A_2\right) < \frac{5}{12} e\left(T \setminus A_0, P\right)$. Then we have
    \begin{align*}
      e\left(T \setminus A_0, S_0(x, 3) \setminus A_2\right) &\leq  \sum_{P \in\mathcal{P}_1} e\left(T \setminus A_0, P\setminus A_2\right) + \sum_{P \in\mathcal{P}_2}e\left(T \setminus A_0, P\setminus A_2\right) \\
                                                             &\leq  |\mathcal{P}|\alpha_3 d+ \sum_{P \in\mathcal{P}} \frac{5}{12} e\left(T \setminus A_0, P\right) \\
                                                             &\leq  4 \alpha_3 K^3 d^3 + \frac{5}{12} e\left(T \setminus A_0, S_0(x, 3)\right).
    \end{align*}
    Using~\eqref{e:edgesTtoS3} and recalling that we are assuming $|T \setminus A_0| \geq d^2/(8K)$, we obtain
    \begin{equation}
      e\left(T \setminus A_0, S_0(x, 3) \setminus A_2\right) \leq \left|T \setminus A_0\right| d \cdot \left( \frac{5}{12} + 32 K^4 \alpha_3 + O(1/d) \right).\label{e:contradictEdgeTwoA0toS3woA2}
    \end{equation}
    If $\alpha_3$ is sufficiently small, however, this contradicts \eqref{e:edgeTwoA0toS3woA2}. Therefore, under the assumption that $\left|T \cap \left(A_3 \setminus A_0\right)\right| \leq 2c\sqrt{d^3 \log d}$, the event $|T \setminus A_0| \geq d^2/(8K)$ and the event that every $P \in \mathcal{P}$ violates \eqref{e:Tsteps13unlikely} cannot occur at the same time, proving that
    \begin{align}
      \Prone[\bigg]{ \left|T \cap \left(A_3 \setminus A_0\right)\right| \leq 2c\sqrt{d^3 \log d} } &\leq \exp\bigl(-\Omega(d \log d)\bigr) + \sum_{P \in \mathcal{P}} \Prone*{ P\text{ satisfies }\eqref{e:Tsteps13unlikely} }.\label{e:1st:reminder}
    \end{align}

    To show that the right-hand side of~\eqref{e:1st:reminder} is small, we will, for each $y \in P$, consider the event $\{y \in A_2\}$, which we denote by $\mathcal{E}_y$.
    By \Cref{l:constant}, for each $y \in P$ we have $\Pr*{\mathcal{E}_y} \geq \frac{3}{4} + o(1) \geq \frac{2}{3}$, and since all vertices in $P$ are at pairwise distance at least six, the events $\left\{\mathcal{E}_y \;:\; y \in P \right\}$ are mutually independent under $\mathbb{P}$.
    However, the event $\mathcal{E}_y$ is \emph{not} independent of $B(x,1) \cap A_0$, and so we cannot easily compare $\Pr*{\mathcal{E}_y}$ and $\Prone*{\mathcal{E}_y}$. In fact, to bound the probability of $P$ satisfying~\eqref{e:Tsteps13unlikely}, it will be useful to treat the set $T \setminus A_0$ as deterministic, and for this reason
    we will further condition on $S(x, 2) \cap A_0$, writing $\mathbb{P}_2$ for the conditional probability space. Thus, our goal is to compare $\Pr*{\mathcal{E}_y}$ and $\Prtwo*{\mathcal{E}_y}$.
    
    To deal with the lack of independence under $\mathbb{P}_2$, we will use the $K$-projection property (\Cref{c:projection}) with $\ell=3$, which implies that for every $y \in P$ there is a subgraph $G(y)\subseteq G$ which is also in $\mathcal{H}(K)$, contains $y$, is disjoint from $B(x, 2)$ and satisfies $d_{G(y)}(w) \geq d_G(w)-3K$ for any $w \in V(G(y))$ (see \Cref{f:projection}).
    \begin{figure}
      \centering
      \begin{tikzpicture}[scale=1, thick, every node/.style={circle}]
        \node (V) at (0,0) [circle,draw, fill, scale=0.6]  {};
        \begin{scope}
            \draw[clip, rotate=90] (0, -3) ellipse (2.5cm and 0.6cm);
            \fill[gray] (2.6,2) -- (2.6,4) -- (3.4,4) -- (3.4,2) -- cycle;
        \end{scope}
        \begin{scope}
            \draw[clip, rotate=90] (0, -4.7) ellipse (3cm and 0.7cm);
            \fill[gray] (4.25,2.3) -- (4.25,5) -- (5.3,5) -- (5.3,2.3) -- cycle;
        \end{scope}
        \begin{scope}
            \draw[clip, rotate=90] (0, -1.5) ellipse (2cm and 0.5cm);
            \fill[gray] (1,1.6) -- (2,1.6) -- (2,2) -- (1,2) -- cycle;
        \end{scope}

        \node[] at (0,0.4) {$x$};
        \draw[rotate=90] (0, -1.5) ellipse (2cm and 0.5cm);
        \node[] at (1.5,2.4) {$N(x)$};
        \draw[rotate=90] (0, -3) ellipse (2.5cm and 0.6cm);
        \node[] at (3,2.9) {$S(x, 2)$};
        \draw[rotate=90] (0, -4.7) ellipse (3cm and 0.7cm);
        \node[] at (4.7,3.4) {$S(x, 3)$};

        \node (Y1) at (4.8,0.4) [circle,draw, fill, scale=0.6] {};
        \node[] at (4.8,0.8) {$y_1$};
        \node (Y2) at (4.8,-1) [circle,draw, fill, scale=0.6] {};
        \node[] at (4.8,-1.4) {$y_2$};

        \path[draw=black] (4.8,0.4) -- (8.8,1.4) -- (12.8,0.4) -- (8.8,-0.6) -- cycle;
        \path[draw=black] (4.8,-1) -- (8.8,0) -- (12.8,-1) -- (8.8,-2) -- cycle;
        \node[] at (8.8,0.4) {$G(y_1)$};
        \node[] at (8.8,-1) {$G(y_2)$};

    \end{tikzpicture}
    \caption{For two vertices $y_1, y_2 \in S_0(x, 3)$, by applying \Cref{c:projection} there exist subgraphs $G(y_1)$ and $G(y_2)$ that are disjoint from $B(x, 2)$ and lie in $\mathcal{H}(K)$.}
    \label{f:projection}
    \end{figure}
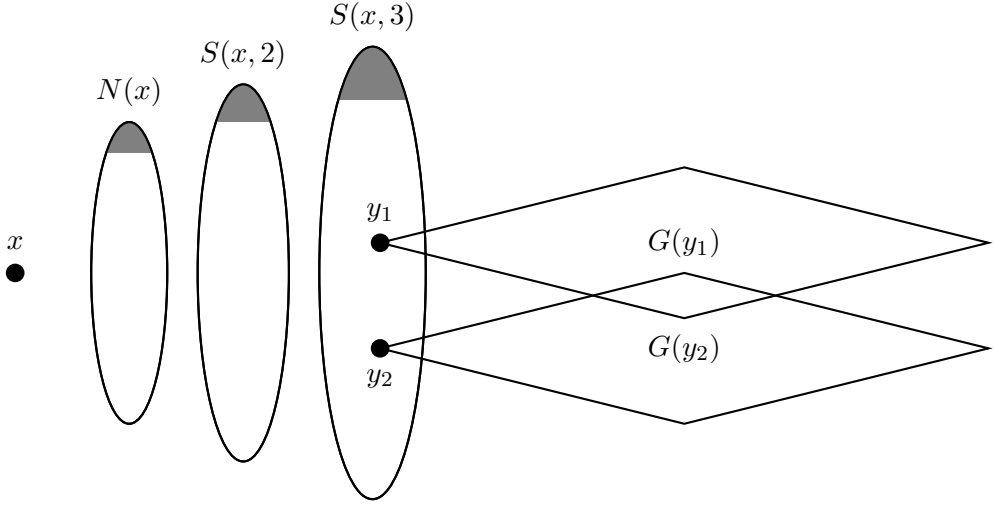

    Recall that $\sigma(x) = \sqrt{\frac{\log d(x)}{d(x)}}$ and observe that
    \[
      \sigma(x) = (1+o(1)) \sqrt{\frac{\log d_{G(y)}(y)}{d_{G(y)}(y)}}.
    \]
    Let us consider the bootstrap percolation process restricted on $G(y)$ with the set of initially infected vertices
    $A^y_0(m) := A_0 \cap V(G(y))$ and infection threshold of $r(w)= \frac{d(w)}{2} + m$ for each $w \in G(y)$. For each $i\in \mathbb N$ we define 
    \[A^y_{i} (m):= A^y_{i-1}(m)\cup \left\{ w \in V(G(y)) \;:\; \left|N_{G(y)}(w) \cap A^y_{i-1}(m)\right| \geq \frac{d_{G(y)}(w)}{2} + m\right\}.\]
    Applying \Cref{l:constant} to $G(y)$ with $m=2K$, we deduce that
    \begin{equation}
      \Prtwo*{ y\in A^y_2(2K) } \geq \frac{3}{4} + o(1) \geq \frac{2}{3}.
    \end{equation}
    However, since $\frac{d_{G(y)}(w)}{2} \geq \frac{d_G(w) - 3K}{2} \geq \frac{d_G(w)}{2} - 2K$ for each $w \in V(G(y))$, it follows that $$A^y_2(2K) \subseteq A_2$$ and we deduce that $ \Prtwo*{y \in A_2} \geq \frac{2}{3}$ for each $y \in P$. Furthermore, as before, since all vertices in $P$ are at pairwise distance at least six, the events $\mathcal{E}_y$ are mutually independent under $\mathbb{P}_2$.

     By \Cref{c:backwards}, we have that each vertex of $P \subseteq S_0(x,3)$ has at most $3K$ neighbours in $T \subseteq S(x,2)$.
    For each $i \in [3K]$, let $b_i$ be the number of elements of $P$ with $i$ neighbours in $T \setminus A_0$.
    Then $e\left(T \setminus A_0,P\right) = \sum_{i=1}^{3K} i b_i$ and the preceding discussion implies that $e\left(T \setminus A_0, P \setminus A_2\right)$ is stochastically dominated by $Y=\sum_{i=1}^{3K} i B_i$, where $B_i \sim \Bin\left(b_i, \frac{1}{3}\right)$.
    Letting \[ \tau \coloneqq \max \left\{ \frac{5}{12}\sum_{i=1}^{3K} i b_i, \alpha_3 d \right\} - \frac{1}{3}\sum_{i=1}^{3K} i b_i \geq \max\left\{\frac{1}{12}\left(\sum_{i=1}^{3K} i b_i\right), \frac{1}{5}a_3d\right\},\] we have by Hoeffding's inequality (the second inequality in \Cref{l:sumBin}) that
    \begin{align*}
      \Prtwo*{P \text{ satisfies } \eqref{e:Tsteps13unlikely}} \leq \Pr*{Y \geq \frac{1}{3} \sum_{i=1}^{3K} i b_i + \tau}
      \leq  \exp\left(-\frac{2\tau^2}{3 \cdot 3K \cdot \Ex{Y}}\right),
    \end{align*}
    where the middle probability is taken with respect to the distribution of $Y$. Using that
    \begin{align*}
     3\cdot\Ex{Y} = \sum_{i=1}^{3K} i b_i \leq 12\tau \qquad \text{ and } \qquad \frac{\tau}{18K} \geq \frac{\alpha_3 d}{90 K} \geq 4 \alpha_2 d,
    \end{align*}
    we then have that $\Prtwo*{P \text{ satisfies \eqref{e:Tsteps13unlikely}}} \leq \exp(-4 \alpha_2 d)$. Since this bound holds for every realisation of $S(x, 2) \cap A_0$, the same inequality holds for $\Prone*{P \text{ satisfies \eqref{e:Tsteps13unlikely}}}$. Recalling that $|\mathcal{P}| =O(d^2)$, we obtain by \eqref{e:1st:reminder} and the union bound that
    \begin{align*}
      \Prone*{\left|T \cap \left(A_3 \setminus A_0\right)\right| \leq  2c\sqrt{d^3 \log d}} \leq \exp\left(-3 \alpha_2 d\right),
    \end{align*}
    proving \Cref{c:Tsteps13}.
  \end{proof}

  We are now ready to finish the proof of \Cref{l:dimred}.
 
\begin{proof}[Proof of \Cref{l:dimred}]
  Note that by \eqref{e:conditionR} and \eqref{e:boundmsum}, we have 
  \[
    \Pr*{ x \not\in A_5 } \leq \exp\bigl(-\alpha_1^2 d\bigr) + \exp\bigl(-\alpha_2d\bigr) \leq \exp\left(-\frac{\alpha_1^2 d}{2}\right),
  \]
  and so \Cref{l:dimred} holds with $\beta =  \alpha_1^2/2$.
\end{proof}

\subsection{Proof of Lemma \ref{l:super-exponential}}

In this section we will prove \Cref{l:super-exponential} about the existence of a constant $\beta >0$ (independent of $x$) such that
  \[
    \Pr[\big]{ x \notin A_{11} } \leq \exp\bigl( - \beta d(x)^2\bigr).
  \]
 Throughout this section we  let $d=d(x)$.
 
   \begin{proof}[Proof of \Cref{l:super-exponential}] 
   Suppose that $x \notin A_{11}$, and let $T_\ell \coloneqq S(x, \ell) \setminus A_{11-\ell}$.   
   By definition, no vertex of $T_\ell$ is infected by time $11 - \ell$. Therefore, each $y \in T_\ell$ has at most $d(y)/2$ neighbours in $A_{11-(\ell+1)}$.
  Hence, by \Cref{c:regular,c:backwards}, each vertex of $T_\ell$ has at least $(d - 3K\ell)/2$ neighbours in $T_{\ell+1}$.
  Moreover, since $G$ has $K$-bounded backwards expansion (\Cref{c:backwards}), every vertex in $S(x,\ell + 1)$ has at most $K(\ell+1)$ neighbours in $S(x,\ell)$.
  Therefore,\[ |T_{\ell+1}| \geq \frac{|T_\ell| \cdot (d - 3K\ell)}{2K(\ell+1)} = \Omega\big( |T_\ell| \cdot d\big). \]
  Since we are assuming $x \notin A_{11}$, we have $|T_0| = 1$, and therefore by induction there exists a constant $\alpha > 0$ such that $|T_\ell| \geq \alpha d^\ell$ for every $0 \leq \ell \leq 6$.
  Hence, if $x \notin A_{11}$, we may take a subset $T' \subseteq T_6 = S(x,6) \setminus A_5$ of size $\alpha d^6$.
  
  By \Cref{l:partitiondist} there is a partition $S(x, 6) = P_1 \cup \dots \cup P_m$ where $m=O\left(d^5\right)$ such that $\dist(y_1,y_2) \geq 12$ for each $j \in [m]$ and distinct $ y_1,y_2 \in P_j$.
  We claim that for $\varepsilon \ll \alpha$ sufficiently small there is some $j\in [m]$ such that
  \begin{equation}\label{e:Bjprop}
    |P_j| \geq \varepsilon d \qquad \text{and} \qquad |P_j \cap T'| \geq \varepsilon |P_j|.
  \end{equation}
  Indeed, $|S(x,6)| = O(d^6)$ by \Cref{c:regular}. Hence, if \eqref{e:Bjprop} fails to hold for every $j \in [m]$, then
  \[
    |T'| \leq \varepsilon d m + \varepsilon |S(x,6)| < \alpha d^6,
  \]
  a contradiction.

  Let $P$ be such that \eqref{e:Bjprop} holds.
  For each $y \in P$, consider the event $\{y \notin A_5\}$, which we denote by $\mathcal{E}_y$.
  The events $\left\{ \mathcal{E}_y \;:\; y \in P\right\}$ are independent, and by \Cref{l:dimred} there is some $\beta'>0$ (independent of $y$) such that $\Pr*{y \not\in A_5} \leq \exp(-\beta' d(y))$ for every $y \in P$.
  Since $d(y) \geq d/2$ for every $y \in P$, we have that $|P \setminus A_5|$ is stochastically dominated by a $\Bin(|P|,p')$ random variable with $p' \coloneqq \exp(-\beta' d/2)$. By \eqref{e:Bjprop} and Chernoff's inequality (\Cref{l:Chernoff}\ref{i:bigtail}), there is some constant $\beta''(\beta',\varepsilon)>0$ such that
  \begin{equation}\label{e:sizeofBi}
    \Pr[\big]{|P \cap T'| \geq \varepsilon |P|} \leq \Pr[\big]{ |P \setminus A_5| \geq \varepsilon |P|} \leq \left(\frac{e p'}{\varepsilon}\right)^ {\varepsilon |P|} \leq \exp\left(-\beta'' d^2\right).
  \end{equation}
  By a union bound over the partition classes $P_j$ ($j \in [m]$), we obtain
  \begin{align*}
    \Pr*{ x \not\in A_{11} } \leq \Pr[\Big]{\text{some } P_j \text{ satisfies } \eqref{e:Bjprop}} \leq m\exp\left(-\beta'' d^2\right) \leq \exp\left(-\frac{\beta'' d^2}{2}\right),
  \end{align*}
  and so the statement holds with $\beta = \beta''/2$.
\end{proof}

\section{Dominating process: stabilisation after two rounds}\label{sec:dominating}

In this section we prove the auxiliary lemmas (\Cref{l:inA3-A2}--\Cref{l:inA2-A1}) needed for the proof of the $0$-statement of \Cref{t:mainThm} in \Cref{s:0-statement}. 
Throughout this section we let $x\in V(G)$, $c > \frac{1}{2}$, $\vartheta(d) = \sqrt{\log d}$,
\[ \sigma(x) = \sqrt{\frac{\log d(x)}{d(x)}},\qquad \gamma(x) = \sqrt{\frac{d(x)}{\vartheta(d(x))}}, \] 
and $p = \frac{1}{2} - c\sigma(x)$. To ease notation let $d=d(x)$, $\sigma=\sigma(x)$ and $\gamma=\gamma(x)$. Assume there is a $K \in \mathbb{N}$ such that $G \in \mathcal{H}(K)$. 

Observe that if $w \in V(G)$ is such that $\dist(x, w)$ is constant, then by \Cref{c:regular} and the asymptotic estimates $\log (d + O(1))/\log d = 1 + O(1/d)$ and $\sqrt{1 + O(1/d)} = 1 + O(1/d)$,
\begin{equation} 
\label{e:gammaconstant}
\gamma(w) = \gamma \cdot \left(1 + O\left(\frac{1}{d(x)}\right)\right) = \gamma + o(1).
\end{equation}

\subsection{Proof of Lemma \ref{l:inA3-A2}}

In this section we will prove \Cref{l:inA3-A2} about the existence of a constant $\beta >0$ (independent of $x$) such that 
  \[ \Pr{x \in \hat{A}_3 \setminus \hat{A}_2} \leq \exp\left(-\beta \gamma(x)^2 \log d(x)\right).\]

  Let $\mathcal{W}$ be the set of pairs $(W_1, W_2)$ with $W_i \subseteq S_0(x, i)$ satisfying
  \[ W_2 \subseteq N(W_1),\qquad |W_1| = \gamma-3K,\qquad |W_2| = (\gamma-3K)^2/2K, \]
  which will be called \emph{witnesses}. The \emph{weight} of a witness $(W_1, W_2)$ is defined as $\zeta(W_2) \coloneqq e(W_2, S_0(x, 3))$. We observe that,
  for every $(W_1, W_2) \in \mathcal{W}$, we have
  \begin{equation}\label{e:witnessobs}
  \zeta(W_2) = |W_2|(d+O(1)),
  \end{equation}
  since every vertex of $W_2$ has degree $d\pm 2K$ by \Cref{c:regular} and every vertex of $S_0(x,2)$ has at most $4K$ neighbours outside $S_0(x,3)$ by Properties \ref{c:backwards} and \ref{c:localconnection}\ref{i:sparse}.

  The definition of witness is motivated by the following claim (see \Cref{f:witness}).

 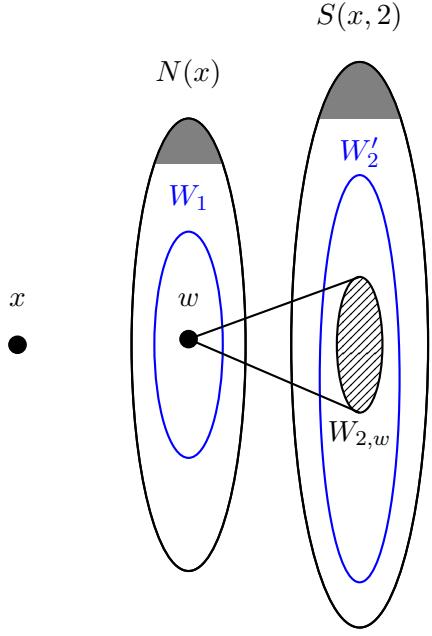
\begin{figure}
 \centering
  \begin{tikzpicture}[scale=1.5, thick, every node/.style={circle}]
        \node (V) at (0,0) [circle,draw, fill, scale=0.6]  {};
        \begin{scope}
            \draw[clip, rotate=90] (0, -3) ellipse (2.5cm and 0.6cm);
            \fill[gray] (2.6,2) -- (2.6,4) -- (3.4,4) -- (3.4,2) -- cycle;
        \end{scope}
        \begin{scope}
            \draw[clip, rotate=90] (0, -1.5) ellipse (2cm and 0.5cm);
            \fill[gray] (1,1.6) -- (2,1.6) -- (2,2) -- (1,2) -- cycle;
        \end{scope}
        
        \node[] at (0,0.4) {$x$};
        \draw[rotate=90] (0, -1.5) ellipse (2cm and 0.5cm);
        \node[] at (1.5,2.4) {$N(x)$};
        \draw[rotate=90] (0, -3) ellipse (2.5cm and 0.6cm);
        \node[] at (3,2.9) {$S(x, 2)$};

        \draw[rotate=90,draw=blue] (0, -1.5) ellipse (1cm and 0.3cm);
        \node[] at (1.5,1.3) {\textcolor{blue}{$W_1$}};
        \node (x) at (1.5,0.05) [circle,draw, fill, scale=0.6]  {};
        \node[] at (1.5,0.4) {$w$};
        \draw[rotate=90, draw=blue] (-0.3, -3) ellipse (1.8cm and 0.35cm);
        \node[] at (3,1.7) {\textcolor{blue}{$W_2'$}};
        \draw[rotate=90,pattern=north east lines] (0, -3) ellipse (0.6cm and 0.2cm);
        \node[] at (3,-0.8) {$W_{2,w}$};

        \path[draw=black] (1.5,0.05) -- (3,0.6);
        \path[draw=black] (1.5,0.05) -- (3,-0.6);
    \end{tikzpicture}
    \caption{For a vertex $x \in (\hat{A}_3 \setminus \hat{A}_2)$ there is a set $W_1 \subseteq S_0(x,1) \cap (\hat{A}_2 \setminus \hat{A}_1)$ of size $\gamma-3K$. For each $w \in W_1$ there is a set $W_{2, w} \subseteq S_0(x, 2) \cap \hat{A}_1 \setminus \hat{A}_0$, and their union is $W_2'$, which contains a subset $W_2$ of size exactly $(\gamma-3K)^2/2K$.}
    \label{f:witness}
    \end{figure}
  
  \begin{claim}\label{c:0-statement-witness}
    If $x \in \hat{A}_3 \setminus \hat{A}_2$, there exists a witness $(W_1, W_2) \in \mathcal{W}$ such that \[
    Z \geq \Ex{Z} + (c \sigma d - 3\gamma)|W_2|,
    \]
    where $Z = Z(W_2) := e\bigl(W_2, S_0(x, 3) \cap \hat{A}_0 \bigr)$.
  \end{claim}

  \begin{proof}[Proof of \Cref{c:0-statement-witness}]
    We start by observing that, by definition of the $\Boot_2(\gamma)$ process, entering $\hat{A}_3$ requires $d/2$ infected neighbours, while entering $\hat{A}_2$ requires only $d/2 - \gamma$ ones. Therefore, the event $\{x \in \hat{A}_3 \setminus \hat{A}_2\}$ implies the existence of a set $W_1' \subseteq S(x,1) \cap (\hat{A}_2 \setminus \hat{A}_1)$ of recently-infected neighbours of $x$ of size $|W_1'|=\gamma$.
    Since $|D(x) \cap N(x)| \leq 1 \leq 3K$ by \Cref{c:localconnection}\ref{i:small}, we may take a subset $W_1 \subseteq W_1' \cap S_0(x,1)$ of size $\gamma - 3K$.
    
    Similarly, since each $w \in W_1$ is in $\hat{A}_2 \setminus \hat{A}_1$, it has $\gamma(w) \geq \gamma - 1$ neighbours in $\hat{A}_1 \setminus \hat{A}_0$, where the inequality uses~\eqref{e:gammaconstant}.
    Moreover, since each $w \in S_0(x, 1)$ has at most $K$ neighbours in $B(x, 1)$ by \Cref{c:backwards} and at most $K$ neighbours in $D(x)$ by \Cref{c:localconnection}\ref{i:sparse}, for each $w \in W_1$ the set $W_{2,w} \coloneqq N(w) \cap S_0(x,2) \cap (\hat{A}_1 \setminus \hat{A}_0)$ has size at least $\gamma-3K$ (see \Cref{f:witness}).
    Moreover, every element of \[ W_2' := \bigcup_{w\in W_1} W_{2,w} \] has at most $2K$ neighbours in $W_1$ by \Cref{c:backwards}, and so $|W_2'| \geq \frac{(\gamma-3K)^2}{2K}$ by double counting. We may therefore choose a $W_2 \subseteq W_2'$ of size $\frac{(\gamma-3K)^2}{2K}$.
    
    It remains to show that $Z(W_2) \geq \Ex{Z(W_2)} + (\sigma d - 3\gamma)|W_2|$. By definition of $\zeta$ and $Z$, we have $\mathbb{E}[Z] = \zeta p$.
    On the other hand, $W_2 \subseteq \hat{A}_1 \setminus \hat{A}_0$ by construction. Therefore, by Properties \ref{c:regular}, \ref{c:backwards} and \ref{c:localconnection}\ref{i:sparse}, every $w \in W_2 \subseteq S_0(x,2)$ has at least $\frac{d(w)}{2} - 2\gamma(w) - 4K = \frac{d}{2} - 2 \gamma + O(1)$ neighbours in $S_0(x,3) \cap \hat{A}_0$.
    Since $p = 1/2 - c\sigma$, using \eqref{e:witnessobs} we obtain
    \[ Z(W_2) \geq |W_2|(d/2 - 2\gamma + O(1)) = \zeta/2 - |W_2|(2\gamma + O(1)) \geq  \zeta p + (c \sigma d - 3\gamma)|W_2|, \]
    where we used that $ \zeta  p = (\zeta/2) - c \sigma |W_2|(d+O(1))$ and $\gamma \gg 1 \gg c\sigma$ in the last inequality.
  \end{proof}

Therefore, if the process has not stabilised after three rounds, there is a witness such that $Z$ exceeds its expectation.
On the other hand, the next claim shows that this is a low probability event. To simplify notation, we set $s \coloneqq (\gamma-3K)^2/2K$ in the remainder of this section. Recall that every witness $(W_1, W_2)$ has $|W_2| = s$.

    \begin{claim}\label{c:0-statement-prob}
    Let $s \in \mathbb{N}$.
    For every witness $(W_1, W_2)$, we have
    \[ \Pr*{Z(W_2) \geq \Ex{Z} + (c \sigma d - 3\gamma)s} \leq \exp\left(-(2+o(1)) c^2 s \log d  \right), \]
      where $Z = Z(W_2) := e\bigl(W_2, S_0(x, 3) \cap \hat{A}_0 \bigr)$.  
  \end{claim}

  \begin{proof}[Proof of \Cref{c:0-statement-prob}]
    Recall that by \Cref{c:backwards}, a vertex in $S_0(x, 3)$ has at most $3K$ neighbours in $S(x, 2)$.
    For $i \in [3K]$, let $X_i$ be the set of elements of $S_0(x,3)$ having $i$ neighbours in $W_2$.
    We have
    \[
      \zeta = \zeta(W_2) = \sum_{i=1}^{3K} i \cdot |X_i|,
    \]
    and letting $M \coloneqq \sum_{i=1}^{3K} i^2 |X_i|$ we have by the first inequality of \Cref{l:sumBin} that, for any $t \geq 0$,
    \begin{equation}\label{e:probwitness}
      \Pr*{Z(W_2) \geq \Ex{Z} + t} \leq \exp\left(-\frac{2t^2}{M}\right).
    \end{equation}

    Let $X_{\geq 2}$ be the set of elements of $S_0(x, 3)$ having at least two neighbours in $W_2$.
    Observe that $|X_{\geq 2}| \leq \binom{|W_2|}{2} \leq \frac{\gamma^4}{2}=o(sd)$, since any two elements of $S_0(x, 2)$ have at most one neighbour in $S_0(x, 3)$ by \Cref{c:localconnection}\ref{i:cherry}.
    Therefore, recalling~\eqref{e:witnessobs}, we have 
    \[ 
    M = \sum_{i=1}^{3K} i^2 |X_i| \leq |X_1| + o(sd) \leq (1+o(1)) \zeta = (1+o(1))sd 
    \] 
    and with $t =  (c \sigma d - 3\gamma)s$ the right hand-side of \eqref{e:probwitness} is
    \[
        \exp\left(- \frac{2s^2(c \sigma d - 3 \gamma)^2}{(1+o(1))sd}\right) \leq \exp\left( - (2+o(1)) s c^2 \sigma^2 d\right) = \exp(-(2+o(1)) c^2 s \log d),
    \]
    since $\sigma d \gg \gamma$. This proves the claim.
  \end{proof}

  We are now ready to prove \Cref{l:inA3-A2}.
  
\begin{proof}[Proof of \Cref{l:inA3-A2}]
  Suppose there is an $x \in \hat{A}_3 \setminus \hat{A}_2$ and consider the witness $(W_1, W_2)$ from \Cref{c:0-statement-witness}. Recall that every witness satisfies $|W_1|=\gamma-3K$ and $|W_2| = s = (\gamma-3K)^2/2K$. There are $\binom{d}{|W_1|}$ choices for $W_1$ and since $|W_1|=\gamma-3K=o(s)$ this number is at most $\exp(o(s \log d))$. 
  Since $W_2 \subseteq N(W_1)$ and $\gamma d/s= \Theta(d/\gamma)= \Theta(\sqrt{d\vartheta})$, for each $W_1 \subseteq N(x)$ there are at most
  \[
   \binom{|N(W_1)|}{s} \leq \binom{\gamma d}{s} \leq \exp\Bigl((1/2+o(1)) s \log d \Bigr)
  \]
  choices for $W_2$ with $|W_2|=s$. Therefore, we may use \Cref{c:0-statement-prob} and the union bound to obtain 
  \[
    \Pr*{x \in \hat{A}_3 \setminus \hat{A}_2} \leq \exp\Bigl((1/2+o(1)) s \log d -(2+o(1)) c^2 s \log d \Bigr).
  \]
  There is some $\varepsilon >0$ such that $c= \frac{1}{2}+ \varepsilon$, and therefore
  \[
    \Pr*{x \in \hat{A}_3 \setminus \hat{A}_2} \leq \exp(-(2\varepsilon +o(1)) s\log d) \leq \exp\left(-\frac{\varepsilon}{2K} \cdot \gamma^2 \log d \right),
  \]
  and so \Cref{l:inA3-A2} holds with $\beta=\frac{\varepsilon}{2K}$.
\end{proof}

 \subsection{Proof of Lemma \ref{l:inA2-A1}}

In this section we will prove \Cref{l:inA2-A1}, claiming 
  \[ \Pr{x \in \hat{A}_2 \setminus \hat{A}_1} \leq \exp\left(-\sqrt{d(x)}\right).\] 
 
 The proof is very similar to the proof of \Cref{l:inA3-A2} in the previous section. As in the proof of \Cref{c:0-statement-witness}, since $x \in \hat{A}_2 \setminus \hat{A}_1$, there exists a set $W' \subseteq N(x) \cap (\hat{A}_1 \setminus \hat{A}_0)$ of size $\gamma$. Take $W$ to be a subset $W' \cap S_0(x, 1)$ of size $\gamma-1$, which is possible by \Cref{c:localconnection}\ref{i:small}, and define 
\[ \zeta' := e(W, S_0(x,2)).\] By Properties \ref{c:backwards} and \ref{c:localconnection}\ref{i:sparse}, each $w \in W \subseteq S_0(x,1)$ has at most $2K$ neighbours outside $S_0(x,2)$. Together with \Cref{c:regular} this implies that 
\begin{equation}\label{e:zetaprime-obs}
\zeta' =  |W|(d+O(1)).
\end{equation}
 Moreover, since every $w \in W$ is infected in the first round, we have
  \[|N(w) \cap S_0(x, 2) \cap \hat{A}_0| \geq d(w)/2 - 2\gamma + O(1) =  d/2 - 2\gamma + O(1),\]
and so
  \[ Z'(W) \coloneqq e(W, S_0(x,2) \cap \hat{A}_0) \geq |W|(d/2 - 2\gamma + O(1)) \geq \zeta' p + (c \sigma d - 3\gamma)|W|, \]
  observing that $\mathbb{E}[Z'] = \zeta' p$.

  By \Cref{c:backwards}, every vertex in $S_0(x,2)$ has at most $2K$ neighbours in $W \subseteq B(x, 2)$. If we let $X'_i$ denote the set of elements of $S_0(x,2)$ with $i$ neighbours in $W$ for $i \in [2K]$, we have that $$\zeta' = \sum_{i=1}^{2K} i|X'_i|.$$ 
Similarly as in the proof of \Cref{c:0-statement-prob}, let $X'_{\geq 2}$ be the set of elements of $S_0(x, 2)$ having at least two neighbours in $W$. We have $|X'_{\geq 2}| \leq \binom{|W|}{2} \leq \gamma^2 = o(\gamma d)$, since any two elements of $S_0(x,1)$ have at most one common neighbour in $S_0(x,2)$ by \Cref{c:localconnection}\ref{i:cherry}, and therefore, by~\eqref{e:zetaprime-obs},
\[
  \sum_{i=1}^{2K} i^2 |X'_i| \leq |X'_1| + o(\gamma d) \leq (1+o(1))\zeta' = (1+o(1))\gamma d.
  \]
  By the first inequality of \Cref{l:sumBin}, for $t = |W|(c \sigma d - 3\gamma)$ we have that
  \begin{equation}\label{eq:philarge}
      \Pr{Z'(W) \geq \Ex{Z'} + t} \leq \exp\left(-\frac{2(\gamma-1)^2(c \sigma d - 3\gamma)^2}{(1+o(1))\gamma d}\right) = \exp\left(-(2+o(1))c^2  \gamma  \log d \right).
  \end{equation}
  Moreover, since $W \subseteq N(x)$ and $d/(\gamma-1)= \Theta(\sqrt{d \vartheta})$, there are at most
  \begin{equation} \label{eq:countW}
    \binom{d}{\gamma-1} \leq \exp\Bigl((1/2+o(1))\gamma \log d\Bigr)
  \end{equation}
  choices for $W$.
  Combining \eqref{eq:philarge} and \eqref{eq:countW} and using the fact that $c>1/2$ we obtain
  \[
    \Pr{x \in \hat{A}_2 \setminus \hat{A}_1} \leq \exp\left( -(1+o(1))2c^2  \gamma  \log d + (1/2 +o(1)) \gamma \log d\right) \leq \exp(- \Omega(\gamma \log d)).
  \]
  Since $\log d \gg \vartheta(d)$, it follows that $\gamma \log d \gg \sqrt{d}$, finishing the proof of \Cref{l:inA2-A1}.
  
\section{Examples of geometric graphs}\label{s:examples}

We dedicate this section to the illustration of the class $\mathcal{H}= \bigcup_{K \in \mathbb{N}} \mathcal{H}(K)$ of high-dimensional geometric graphs by giving some examples of graph classes that are contained in $\mathcal{H}$.
In addition to the formal proofs that the required properties are satisfied, we provide some intuition about the \emph{local coordinate system} that makes these graphs \emph{high-dimensional}.

\subsection{Cartesian Product graphs}\label{s:cartesian}
\begin{definition}
For $n\in \mathbb N$    let $(H_i)_{i \in [n]}$ be a sequence of connected graphs, called \emph{base graphs}.
  We define the \emph{(Cartesian) product graph} $G= \car_{i=1}^n H_i$ as the graph with vertex set
  \[V(G) \coloneqq \prod_{i=1}^n V(H_i) = \{(x_1,\dots, x_n) \;:\; x_i \in V(H_i) \text{ for all } i \in [n]\}\]
  and with edge set
  \[
    E(G) \coloneqq \left\{ \bigl\{ x,y \bigr\} \;:\; \text{there is some } i \in [n] \text{ such that } \{x_i,y_i\} \in E(H_i) \text{ and } x_j=y_j \text{ for all }j \neq i \right\}.
  \]
\end{definition} 
In the case of product graphs, there is a clear coordinate system coming from the product structure.

\begin{lemma}\label{l:propproduct}
  Let $C \in \mathbb{N}$ be such that $C>1$ and let $H_1,\ldots, H_n$ be graphs such that $1< |H_i| \leq C$ for all $i \in [n]$.
  Then $G_n= \car_{i=1}^n H_i \in \mathcal{H}(C)$ for all $n\in \mathbb N$.
\end{lemma}
\begin{proof}
Write $G\coloneqq G_n$. Given distinct vertices $x,y \in V(G)$, let us define
\[
I(x,y) \coloneqq \{ i \in [n] \;:\; x_i \neq y_i\},
\]
noting that $1\leq|I(x,y)| \leq \dist(x,y)$.

\begin{proof}[Proof of \Cref{c:regular}]\let\qed\relax
This property of (Cartesian) product graphs was already observed by Lichev \cite{Lichev2021TheGC}. Indeed, given $x \in V(G)$, $\ell \in \mathbb{N}$ and $y \in S(x,\ell)$, we see that
\[
|d(x) - d(y)| = \left| \sum_{i \in [n]} d_{H_i}(x_i) - \sum_{i \in [n]} d_{H_i}(y_i)\right| =  \left| \sum_{i \in I(x,y)} (d_{H_i}(x_i) -  d_{H_i}(y_i))\right| \leq C  |I(x,y)| \leq C  \ell.
\]
\end{proof}

\begin{proof}[Proof of \Cref{c:backwards}]\let\qed\relax
Given $x \in V(G)$, $\ell \in \mathbb{N}$ and $y \in S(x,\ell)$, the neighbours of $y$ in $B(x,\ell)$ must differ from $y$ in coordinates in $I(x,y)$. Hence there are at most $C \ell$ neighbours of $y$ in $B(x,\ell)$ as $|I(x, y)| \leq \ell$.
\end{proof}

\begin{proof}[Proof of \Cref{c:localconnection}]\let\qed\relax
Given $x\in V(G)$ let
\[
D(x) \coloneqq \{ y \in V(G) \;:\; |I(x,y)| \neq \dist(x,y)\},
\]
so that for all $\ell \in \mathbb{N}$
\[
S_0(x,\ell) \coloneqq S(x,\ell) \setminus D(x) = \{y \in S(x,\ell) \;:\; |I(x,y)| =\ell\}.
\]

Since a vertex in $S(x,\ell) \cap D(x)$ differs from $x$ in at most $\ell-1$ coordinates it is clear that $|S(x,\ell) \cap D(x)| \leq C^{\ell-1} \binom{n}{\ell-1}\leq C^{\ell-1}d(x)^{\ell-1}$ since $d(x) \geq n$, and so \ref{i:small} holds.
Furthermore, if $y \in S_0(x,\ell)$, then any $z \in N(y)$ obtained from $y$ by changing a coordinate outside $I(x, y)$ satisfies $\dist(x, z) = |I(x, z)|$; therefore, any $z \in N(y) \cap D(x)$ must differ from $y$ in some coordinate from $I(x,y)$. Since there are $\ell$ such coordinates, \ref{i:sparse} holds.
Finally, if $w,y \in S_0(x,\ell)$ are distinct, and have a common neighbour $z \in S_0(x,\ell+1)$, then it is easy to verify that $|I(x,w) \cup I(x,y)| = \ell+1$ and hence there is a unique common neighbour $z \in S_0(x,\ell+1)$ which agrees with $w$ on $I(x,w)$, with $y$ on $I(x,y)$ and with $x$ on $[n] \setminus (I(x,w) \cup I(x,y))$. Hence \ref{i:cherry} holds.
\end{proof}

\begin{proof}[Proof of \Cref{c:projection}]\let\qed\relax
  We induct on the dimension $n$ of the product graph $G_n$. For the base case, it is easy to verify that if $|V(G)| \leq C$, then $G \in \mathcal{H}(C)$. Let $x \in V(G), \ell \in \mathbb{N}$ and $y \in S(x,\ell)$. Let $G(y) \coloneqq \left(\car_{i \not\in I(x,y)} H_i\right) \car \left( \car_{i \in I(x,y)} \{y_i\}\right)$. Clearly $y \in V(G(y))$ and $V(G(y)) \cap B(x, \ell-1)=\emptyset$.

Furthermore, for any $w \in V(G(y))$,
\[
|d_{G(y)}(w) - d_G(w)| = \left|\sum_{i \not\in I(x,y)} d_{H_i}(w_i) - \sum_{i \in [n]} d_{H_i}(w_i)\right| = \sum_{i \in I(x,y)} d_{H_i} (w_i) \leq C |I(x,y)| \leq C \ell.
\]

Finally, since $|I(x,y)| \geq 1$, $G(y)$ is a product graph of dimension at most $n-1$, and so by the induction hypothesis $G(y) \in \mathcal{H}(C)$.
\end{proof}

\begin{proof}[Proof of \Cref{c:partitioncond}]\let\qed\relax
Let $x \in V(G)$, $\ell \in \mathbb{N}$ and $y \in S_0(x,\ell)$. Let
\[
Z \coloneqq B(y, 2\ell-1) \cap S_0(x, \ell) =\{ z \in S_0(x,\ell) \;:\; \dist(z,y) \leq 2\ell-1\}.
\]

Note that, for any $x\in V(G)$, $y \in S_0(x, \ell)$ and $z \in S_0(x, \ell)$ we have $I(x,z) \symdif I(x,y) \subseteq I(y,z)$
and since $|I(x,z)| = |I(x,y)| = \ell$, we have $|I(x,z) \symdif I(x,y)| = 2 |I(x, z)\setminus I(x, y)|$. Hence,
\begin{equation}\label{e:distances}
\dist(y,z) \geq |I(y, z)| \geq |I(x, z)\symdif I(x, y)| = 2 |I(x, z)\setminus I(x, y)|.
\end{equation}
Rearranging this we have
\[
|I(x,z) \setminus I(x,y)| \leq \frac{1}{2} \left(\dist(y,z)\right) <\ell,
\]
for every $z \in Z$. And hence,
\[
Z \subseteq Z' \coloneqq \{ z \in S(x, \ell) \;:\; |I(x,z) \setminus I(x,y)| \leq \ell-1\}.
\]
The choices for $z \in Z'$ may be bounded by first choosing $I(x,z)$ of size $\ell$ with $I(x,z) \cap I(x,y) \neq \emptyset$ and then choosing the values for the coordinates in $I(x,z)$. Therefore,
\[
|Z'| \leq \ell \binom{n-1}{\ell-1}C^{\ell} \leq \ell C^{\ell}d(x)^{\ell-1}
\]
since $d(x) \geq n$.
\end{proof}

\begin{proof}[Proof of \Cref{c:bound}]\let\qed\relax
Finally, it is clear that $|V(G)| \leq C^n$ and $\delta(G) \geq n$, and hence $|V(G)| \leq \exp(\log C \cdot \delta(G)) \leq \exp( C \delta(G)) $.
\end{proof}
This concludes the proof of \Cref{l:propproduct}.
\end{proof}
\begin{remark}\label{r:cube}
In particular, the $n$-dimensional hypercube is in $\mathcal{H}(2)$ and we can pick $D(x)=\emptyset$ for any $x \in V(Q^n)$ in this case.
\end{remark}

Note that, under the condition of \Cref{l:propproduct}, all degrees in $G_n$ are in $\Theta(n)$, and so $\delta(G_n)$ and $\Delta(G_n)$ differ by at most a constant multiplicative factor.
In particular, \Cref{t:mainThm} implies that for \emph{any} such product graph, the critical window has width $O\left( \sqrt{\frac{\log n}{n}} \right)$.

\subsection{The middle layer and odd graphs}
  \begin{definition}
      The \emph{middle layer graph} $M_n$ is the graph whose vertices are all vectors of length $2n-1$ where either $n$ or $n-1$ many entries are $1$, while all other entries are $0$. Two vertices are connected if they differ in exactly one coordinate.
  \end{definition}
We note that the middle layer graph is not a (Cartesian) product graph, since it can be seen to have girth six.
Note that the middle layer graph is a subgraph of a hypercube, i.e., $M_n \subseteq Q_{2n-1}$, and again here there is a clear coordinate system underlying the graph structure.
Furthermore, $M_n$ is an isometric subgraph of $Q_{2n-1}$, i.e., $\dist_{M_n}(u, v)= \dist_{Q_{2n-1}}(u, v)$ for all $u, v \in V(M_n)$. Since $Q_{2n-1} \in \mathcal{H}(2)$, this allows us to transfer some properties from the hypercube directly.

\begin{lemma}\label{l:propmiddle}
  The middle layer graph $M_n \in \mathcal{H}(4)$ for all $n\in \mathbb N$.
\end{lemma}
\begin{proof}
Let $G=M_n$. Given distinct vertices $x,y \in V(G)$, let us define
\[
I(x,y) \coloneqq \{ i \in [2n-1] \;:\; x_i \neq y_i\},
\]
noting that $|I(x,y)| = \dist(x,y)$.

\begin{proof}[Proof of \Cref{c:regular}]\let\qed\relax
  $G$ is $n$-regular.
\end{proof}

\begin{proof}[Proof of \Cref{c:backwards}]\let\qed\relax
The property is preserved under taking isometric subgraphs.
\end{proof}

\begin{proof}[Proof of \Cref{c:localconnection}]\let\qed\relax
  For every $x\in V(G)$ let $D(x) = \emptyset$, so that for all $\ell \in \mathbb{N}$ we have $S_0(x,\ell) = S(x,\ell)$ and thus \ref{i:small} and \ref{i:sparse} hold trivially. Since \ref{i:cherry} is satisfied by the hypercube $Q_{2n-1}$ with $D(x)=\emptyset$, and is preserved under taking isometric subgraphs, it also holds in $M_n$ with $D(x)=\emptyset$.
\end{proof}

\begin{proof}[Proof of \Cref{c:projection}]\let\qed\relax
  We induct on $n$. Let $x \in V(G), \ell \in \mathbb{N}$ and $y \in S(x,\ell)$. Assume first that $\ell$ is even. Let $V(y) = \prod_{i \in I(x, y)} \{y_i\} \times \prod_{i \notin I(x, y)} \{0,1\}$ and set $G(y)=G[V(y)]$, i.e., we fix the coordinates where $y$ differs from $x$ and let the other coordinates vary. It is easy to verify that $G(y)$ is isomorphic to $M_{n-\ell/2}$ and so by the induction hypothesis, we have $G(y) \in \mathcal{H}(4)$. Furthermore, since $M_n$ is $n$-regular and $M_{n-\ell/2}$ is $(n-\ell/2)$-regular, it follows that for any $w \in V(G(y))$, we have $|d_{G(y)}(w) - d_G(w)| = \ell/2$. Finally, it is simple to check that $y \in V(G(y))$ and $V(G(y)) \cap B(x, \ell-1)=\emptyset$.

  The case when $\ell$ is odd is similar: we fix an extra coordinate to balance the number of fixed $0$ coordinates with the number of fixed $1$ coordinates.
\end{proof}

\begin{proof}[Proof of \Cref{c:partitioncond}]\let\qed\relax
  The property is satisfied by the hypercube $Q_{2n-1}$ with $D(x)=\emptyset$ and $K=2$, and is preserved under taking isometric subgraphs. However, the degrees of vertices are not the same in both of these graphs.
  Nevertheless, since $Q_{2n-1}$ is $(2n-1)$-regular and $M_n$ is $n$-regular, it follows that for each $y \in S_{M_n}(x,\ell)$,
  \[
  |B_{M_n}(y, 2\ell-1) \cap S_{M_n}(x, \ell)| \leq \ell 2^{\ell} (2n-1)^{\ell-1} \leq \ell 4^{\ell} d_{M_n}(y)^{\ell-1}.
\]
\end{proof}

\begin{proof}[Proof of \Cref{c:bound}]\let\qed\relax
  Since $V(M_n) \subseteq V(Q_{2n-1})$, we have $|V(M_n)| \leq 2^{2n-1}$ and it follows that $|V(G)| \leq \exp(4 \delta(M_n))$.
\end{proof}
This concludes the proof of \Cref{l:propmiddle}.
\end{proof}

We also consider the \emph{odd graph}, which is a special \emph{Kneser graph}.
\begin{definition}
For $n,k\in \mathbb N$ with $n\geq k$ the \emph{Kneser graph} $K(n, k)$ is the graph with vertex set $V(K(n, k)) = [n]^{(k)} = \binom{[n]}{k}$ consisting of all the $k$-element subsets of $[n]$, where two $k$-element subsets are adjacent in $K(n, k)$ if they are disjoint.
\end{definition}

Instead of directly proving that $O_n \coloneqq K(2n-1, n-1)$ lies in $\mathcal{H}$ we will use the fact, which is essentially folklore, that $M_n$ and $O_n$ are locally isomorphic as follows.
\begin{lemma}
  Let $x \in [2n-1]^{(n-1)}$ and let $\ell < n-1$.
  Then $B_{O_n}(x, \ell) \cong B_{M_n}(x, \ell)$.
\end{lemma}
\begin{proof}
  Let $x \in V(O_n)$, so $x$ is a subset of $[2n-1]$ of size $n-1$.
  We define a map between the balls $B_{O_n}(x,\ell)$ and $B_{M_n}(x, \ell)$. For each $k \in \bigl[\lfloor \frac{\ell}{2} \rfloor\bigr]$ and $y \in S_{O_n}(x, 2k)$ we map the set $y$ to its indicator function in $\{0,1\}^{2n-1}$. For every $k \in \bigl[\lfloor \frac{\ell-1}{2} \rfloor\bigr]$ and $y \in S_{O_n}(x, 2k+1)$ we map $y$ to the indicator function of its complement $[2n-1] \setminus y$. Note that the complement has size $n$ and hence corresponds to a vertex in $V(M_n)$.

The map obtained in this way is indeed an isomorphism: If $y, z \in B_{O_n}(x, \ell)$ are such that $\{y, z\} \in E(O_n)$, then $y \cap z=\emptyset$. Since the odd girth of $O_n$ is $2n-1$, without loss of generality $y \in S_{O_n}(x, 2k)$ for some $k \in \bigl[\lfloor \frac{\ell}{2} \rfloor\bigr]$ and $z \in S_{O_n}(x, 2k'+1)$ for some $k' \in \bigl[\lfloor \frac{\ell-1}{2} \rfloor\bigr]$, where $k' \in \{k-1, k\}$. But then $y \subseteq [2n-1] \setminus z$ and hence the image of the edge $\{y, z\}$ under the mapping is an edge in $E(M_n)$.
A similar argument shows that every edge in $B_{M_n}(x, \ell)$ is the image of an edge in $B_{O_n}(x, \ell)$ under the mapping and hence, this mapping gives an isomorphism from $B_{O_n}(x, \ell)$ to $B_{M_n}(x, \ell)$ for any $\ell < n-1$.
\end{proof}

So the odd graph $O_n$ is `locally' isomorphic to a graph, i.e., the middle layer graph $M_n$, which we have already shown lies in $\mathcal{H}$. 

As \Cref{c:regular,c:backwards,c:localconnection,c:partitioncond} are all of local nature, they continue to hold in $O_n$ for $\ell  < n-1$, which is already sufficient for our proof. 
Furthermore, since the diameter of $O_n$ is $n-1$, there is only one further case to consider, in which all of the properties above are trivially satisfied with say $K=4$, for which also \Cref{c:bound} trivially holds.

Finally, it remains to show that \Cref{c:projection} holds in $O_n$.
Let $x \in V(O_n)$, $\ell \in \mathbb{N}$ and $y \in S(x, \ell)$.

Assume first that $\ell=2k$ is even. By the structure of the odd graph, the sets $a \coloneqq x \setminus y$ and $b \coloneqq y \setminus x$ have size $k$.
Pick a set $a' \subseteq x \cap y$ and a set $b' \subseteq [2n-1]\setminus (x \cup y)$, both of size $k$
and set 
\[
V(y)\coloneqq\left\{v \in V(O_n)\,  \middle| \begin{array}{l} \big((b \cup a' \subseteq v) \wedge ((b' \cup a)\cap v=\emptyset)\big) \; \vee \\ \big((b' \cup a \subseteq v) \wedge ((b \cup a')\cap v=\emptyset)\big)\; \phantom{\vee}
\end{array}\right\}.
\]
Let $G(y)$ be the induced subgraph of $O_n$ on the vertex set $V(y)$.
Then $y \in V(G(y))$ as $b\cup a' \subseteq y \text{ and } (b' \cup a)\cap y=\emptyset$.
For each vertex $w \in V(y)$ we obtain $|w \cap x| \geq k$ 
and $|w \setminus x|\geq k$. The former inequality implies that $\dist(x, w) \geq 2k+1$ if $\dist(x, w)$ is odd, and the latter implies that $\dist(x, w) \geq 2k$ otherwise, and thus $w \notin B(x, \ell-1)$.
Furthermore, $G(y)$ can be seen to be isomorphic to $M_{n-\ell}$ under the mapping which takes each $w \in V(y)$ with $b' \cup a \subseteq w \text{ and } (b \cup a')\cap w=\emptyset$ to its complement.
Thus by \Cref{l:propmiddle} we have $G(y) \in \mathcal{H}(4)$ and since $G$ is $n$-regular and $M_{n-\ell}$ is $(n-\ell)$-regular we get $|d_{G(y)}(w)-d_G(w)| =\ell \leq 4 \ell$ for each $w \in V(y)$.

The case when $\ell=2k+1$ is similar. In this case, $a \coloneqq x \cap y$ is a set of size $k$ and $b \coloneqq [2n-1]\setminus (x \cup y)$ is a set of size $k+1$.
Pick a set $a' \subseteq x \setminus y$ of size $k$ and a set $b' \subseteq y \setminus x$ of size $k+1$, and consider the set $V(y)$ as defined above.
Taking $G(y)$ to be the induced subgraph of $O_n$ on $V(y)$, we note that $y \in V(G(y))$ as $b' \cup a \subseteq y \text{ and } (b \cup a')\cap y=\emptyset$. Again, for each vertex $w \in V(y)$ we obtain $|w \cap x| \geq k$ 
and $|w \setminus x|\geq k+1$ and thus $w \notin B(x, \ell-1)$. As before, $G(y)$ can be seen to be isomorphic to $M_{n-\ell}$ under the mapping which takes
each $w \in V(y)$ with $b \cup a' \subseteq w \text{ and } (b' \cup a)\cap w=\emptyset$ to its complement and the degrees in $G$ and $G(y)$ differ by $\ell$.

\begin{lemma}\label{l:propodd}
  The odd graph $O_n \in \mathcal{H}(4)$ for all $n\in \mathbb N$.
\end{lemma}

\subsection{The folded hypercube}

\begin{definition}
For $n\in \mathbb N$ the \emph{folded hypercube} $\tilde{Q}_n$ is a graph on the same vertex set as the $(n-1)$-dimensional hypercube, i.e., $V(\tilde{Q}_{n})= \{0,1\}^{n-1}$. It is obtained from $Q_{n-1}$ by joining each vertex $x$ to its antipodal vertex $\tilde{x}$, where $\tilde{x}$ is the vertex differing from $x$ in every coordinate. That is, $\tilde{Q}_n$ is obtained by adding all edges $\{x, \tilde{x}\}$ to the edge set of $Q_{n-1}$.
\end{definition}

We can think of the underlying coordinate system here as being $\{0,1\}^{n}$, where the first $n-1$ coordinates represent the `actual' vertex $v \in V(Q_{n-1})$ and the final coordinate represents movement to the antipodal point.

We note that it is relatively easy to see that the folded hypercube is not expressible as the product of bounded order graphs. Indeed, as we will see shortly, $\tilde{Q}_n$ is locally isomorphic to $Q_n$, and so every connected subgraph of $\tilde{Q}_n$ of bounded order is a subgraph of a hypercube. However, it is easy to show that any product graph in which every factor is a subgraph of a hypercube is at most as dense as a hypercube, and $\tilde{Q}_n$ has strictly more edges than $Q_{n-1}$. In fact, with more care, it can be shown that $\tilde{Q}_n$ is not expressible as a Cartesian product even with factors of unbounded order.

As with $M_n$ and $O_n$, the folded hypercube is `locally' isomorphic to a graph, i.e., the hypercube $Q_n$, which we have already shown lies in $\mathcal{H}$. Indeed the following is shown in \cite{KKM99}.

\begin{lemma}
  Let $x \in \{0,1\}^{n-1}$ and $y \in \{0,1\}^n$ and let $\ell < \left \lfloor \frac{n}{2} \right\rfloor$. Then $B_{\tilde{Q}_n}(x, \ell) \cong B_{Q_n}(y, \ell)$.
\end{lemma}

As before, since by \Cref{r:cube} we have $Q_n \in \mathcal{H}(2)$, and \Cref{c:regular,c:backwards,c:localconnection,c:partitioncond} are all of local nature, they continue to hold in $\tilde{Q}_n$ for $\ell  < \left \lfloor \frac{n}{2} \right\rfloor$, which is already sufficient for our proof. Furthermore, since the diameter of $\tilde{Q}_n$ is $\left \lfloor \frac{n}{2} \right\rfloor$, there is only one further case to consider, in which all of the properties are trivially satisfied with say $K=3$, for which also \Cref{c:bound} trivially holds.

Finally, it is a simple exercise to show that \Cref{c:projection} holds in $\tilde{Q}_n$. Indeed, we note that every vertex $y$ which is at distance $\ell$ from a vertex $x$ either differs from $x$ in exactly $\ell$ coordinates, or differs from $\tilde{x}$ in exactly $\ell-1$ coordinates. In either case we can fix a set $I$ of coordinates of size $2\ell -1$ on which $y$ differs from $x$ in $\ell$ coordinates and from $\tilde{x}$ in $\ell-1$ coordinates. In particular, the subcube $X = \{ z : z_i = y_i \text{ for all } i \in I\}$ contains $y$ and is disjoint from $B_{\tilde{Q}_n}(x,\ell-1)$. Furthermore, since $Q_{n-1} \subseteq \tilde{Q}_n$, it follows that we can take $G(y) = Q_{n-1}[X] \cong Q_{n-2\ell}$, where $Q_{n-2\ell} \in \mathcal{H}(2)$. Indeed, \Cref{c:projection} \ref{i:P4inclusion}-\ref{i:P4intersection} follow by the above, and \Cref{c:projection} \ref{i:P4degree} follows since $\tilde{Q}_n$ is $n$-regular and $Q_{n-2\ell}$ is $(n- 2\ell)$-regular.

\begin{lemma}\label{l:propfolded}
  The folded hypercube $\tilde{Q}_n \in \mathcal{H}(3)$ for all $n\in \mathbb N$.
\end{lemma}

\section{Discussion}\label{s:discussion}

In this paper we extended the results of Balogh, Bollob\'{a}s and Morris \cite{BaBoMo2009} to a large class of high-dimensional geometric graphs. It is perhaps natural to ask what the limits of our proof methods are, and what structural conditions are necessary to exhibit a similar behaviour in terms of the location and width of the critical window, which in both cases seem to be controlled by the local structure of the graphs. 
\begin{question}\label{q:universal}
  Let $G$ be an $n$-regular graph. Under what assumptions are the first two terms in the expansion of the critical probability of majority bootstrap percolation in $G$ given by $\frac{1}{2} - \frac{1}{2}\sqrt{\frac{\log n}{n}}$?
\end{question}

Our proof methods seem to rely on the high-dimensional geometric structure of the graphs, however there are many other graphs which seem inherently high-dimensional which our results do not cover.
For example, the \emph{halved cube} is the graph on $\{0,1\}^n$ where two vertices are adjacent if they have Hamming distance exactly two.
Geometrically, this is the $1$-skeleton of the polytope constructed from an \emph{alternation} of a hypercube.
Whilst this graph is regular and highly symmetric, the presence of many triangles, and indeed large cliques in every neighbourhood, cause the property of bounded backwards expansion (\Cref{c:backwards}) to fail and so the graph does not lie in the class $\mathcal{H}$ as defined in \Cref{s:conditions}.
Furthermore, motivated by the case of the middle layer graph and the odd graph, it would be interesting to know if similar behaviour is present in (bipartite) Kneser graphs with a larger range of parameters or in other graphs arising from intersecting set systems such as Johnson graphs. Here, whilst the graphs still display a sort of fractal self-symmetry, an issue arises with the quantitative aspects of \Cref{c:projection}, since the degrees in these projections drop too quickly and this effect dominates the variations in the number of infected vertices in a vertex's neighbourhood. Another interesting example comes from the \emph{permutahedron}, a well-studied graph which like the hypercube has many equivalent `high-dimensional' representations. Here, whilst \Cref{c:regular,c:backwards,c:localconnection,c:projection,c:partitioncond} are satisfied, the graph is too large to satisfy \Cref{c:bound}, superexponential in its regularity. In \cite{CEGK25} we show that this difficulty can be overcome, and the analogue of \Cref{c:mainThm} also holds in the permutahedron.

In comparison to Theorem \ref{t:BBM}, we give a weaker bound on the width of the critical window. As pointed out in \Cref{r:finer1statement}, our methods recover the upper bound in Theorem \ref{t:BBM}, that is, if $\lambda > \frac{1}{2}$ and
\[
p=\frac{1}{2}-\frac{1}{2} \sqrt{\frac{\log \Delta(G_n)}{\Delta(G_n)}} + \frac{\lambda \log \log \Delta(G_n)}{\sqrt{\Delta(G_n) \log \Delta(G_n)}},
\]
then $\lim_{n \to \infty} \Phi(p,G_n) =1$. However, in the $0$-statement slightly stronger structural assumptions and more delicate counting arguments are required to further bound the width in this manner. Since some of these structural assumptions do not hold in all of our examples (for example in $M_n$ and $O_n$), and also for ease of presentation, we have given a simpler exposition which just determines the first two terms in the expansion of the critical probability. However, since these extra assumptions and counting arguments were necessary in order to determine the threshold for the permutahedron, the details of how to recover the stronger bound on the width of the critical window can be found in \cite{CEGK25}.

In the light of the dependence of our bounds on the minimum and maximum degree, it is also interesting to ask what can happen when the host graph is irregular. Whilst Theorem \ref{t:mainThm} implies that the critical window still does not contain $p=\frac{1}{2}$ for graphs in $\mathcal{H}$, it is not clear if the second term in the expansion of the critical probability is controlled by the maximum degree, the minimum degree or the average degree of the graph (if by any of them). The particular case where $G$ is a product of many \emph{stars}, is considered in \cite{CEGK25}, where we show that the upper bound in Theorem \ref{t:mainThm} is in fact tight, but it is not clear if this behaviour is universal even in irregular product graphs.

It is perhaps also interesting to consider how the process evolves with other non-constant infection thresholds.
For example, if we take $r(v) = \alpha \cdot d(v)$ for some constant $\alpha > 0$, then the arguments in this paper (with the application of Hoeffding's inequality in~\eqref{e:probwitness} replaced by an application of \cite[Theorem 2.7]{McDiarmid}) show that for $d$-regular graphs in $\mathcal{H}$ the $r$-neighbour bootstrap percolation process undergoes a similar transition around
\[
  p = \alpha - \sqrt{\alpha (1-\alpha)} \cdot \sqrt{\frac{\log d}{d}},
\]
and likely this will remain true when $\alpha$ is a very slowly shrinking function of $d$.
It would be interesting to consider what happens for smaller $\alpha$, for example $d^{-\beta}$ for some $0< \beta < 1$.
\begin{question}
  Let $G$ be a $d$-regular high-dimensional graph and consider the $r$-neighbour bootstrap percolation process where $r = \sqrt{d}$.
  What can we say about the location and width of the critical window?
\end{question}

In another direction, for small constant values of the infection threshold $r$, the average case behaviour of the $r$-neighbour bootstrap percolation process on the hypercube has also been considered, where in particular for $r=2$ a threshold for percolation was determined by Balogh and Bollob\'{a}s \cite{BaBo2006HypercubeR2} and a sharper threshold was shown by Balogh, Bollob\'{a}s and Morris \cite{BaBoMo2010HighDimGrid}.
It would be interesting to know if a corresponding threshold could be determined in other high-dimensional graphs, see \cite{KMM24} for an analysis of bootstrap percolation on Hamming graphs.
Moreover, even for the hypercube $Q_n$, very little is known about the typical behaviour of the process for $r \geq 3$, see for example \cite[Conjecture 6.3]{BaBoMo2010HighDimGrid}.

The majority bootstrap percolation process has also been studied in the binomial random graph $G(m,q)$ \cite{HJK17,SV15}.
Perhaps surprisingly, in \cite{HJK17} it is shown that if the edge probability $q$ is close to the connectivity threshold, then the first two terms in the expansion of the critical probability for $G(m,q)$ agree with those given by Bollob\'{a}s, Balogh and Morris for the hypercube  $Q_n$ in \Cref{t:BBM}, taking $n = (m-1)q$ to be the expected degree of a vertex.
The authors conjecture that this behaviour is in fact universal to all approximately regular graphs with a similar density to the hypercube and sufficiently strong expansion properties. In this context, it would be interesting to study the majority bootstrap percolation process on random subgraphs of high-dimensional graphs, in particular of the hypercube, above the connectivity threshold.

\begin{question}
Let $(Q_n)_q$ denote a random subgraph  of the hypercube $Q_n$, in which each edge of  $Q_n$ is  retained independently with probability $q\in (0,1)$. 
 What is the critical probability for majority bootstrap percolation on $(Q_n)_q$ for $q> \frac{1}{2}$? 
\end{question}

\subsection*{Acknowledgements}
We thank the anonymous reviewers for their helpful comments and suggestions. This research was funded in whole or in part by the Austrian Science Fund (FWF) [10.55776/DOC183, 10.55776/F1002, 10.55776/P36131].
This study was financed in part by the Coordenação de Aperfeiçoamento de Pessoal de Nível Superior -- Brasil (CAPES) -- Finance Code 001.
M.C.\ was supported by CNPq (420838/2025-2) and FAPESP (2023/03167-5, 2024/13859-4).
For open access purposes, the author has applied a CC BY public copyright license to any author accepted manuscript version arising from this submission.

\printbibliography

\appendix

\section{Proof of \Cref{r:finer1statement}} \label{s:appendix}

Let us prove that for $\lambda > 1/2$ and
\[
p \geq \frac{1}{2}- \frac{1}{2}\sqrt{\frac{\log d(x)}{d(x)}} + \frac{\lambda \log \log d(x)}{\sqrt{d(x) \log d(x)}}
\]
the statement of \Cref{l:constant} is still true. We follow its proof closely and turn our attention to the affected calculations only.

As before set $X_0=N(x) \cap A_0$ and $X_1= N(x) \cap (A_1(m) \setminus A_0)$ and begin by estimating the conditional expectation of $|X_1|$, where the conditioning is as before.
\begin{claim}\label{c:expapp}
    \[\Exs[\big]{|X_1| }  = \Ex[\big]{|X_1|  \mid (x \notin A_0) \wedge \mathcal{E}_0} = \Omega\left(d(x)^{1/2} (\log d(x))^{2 \lambda -1/2}\right).
    \]
  \end{claim}
\begin{proof}[Proof of \Cref{c:expapp}]
  Let us suppose that $x \notin A_0$ and let $y \in N(x)$. We have
  \begin{align}\label{e:splitprobyinsapp}
    \Prs[\Big]{y \in X_1} = \Prs{y \notin A_0}\cdot \Prs[\Big]{|N(y) \cap A_0| \geq \frac{d(y)}{2}+m \given y \notin A_0}.
  \end{align}

  To bound the second term in \eqref{e:splitprobyinsapp}, we note that $\big|N(y) \cap \bigl(\{x\} \cup N(x)\bigr)\big| \leq K$ by \Cref{c:backwards}.
  Furthermore, we have $d(y) \geq d(x) - K$ since $G$ is $K$-locally almost regular by \Cref{c:regular}.
  Hence, conditioned on the events $\{x, y \notin A_0\}$ and $\mathcal{E}_0$, $|N(y) \cap A_0|$ stochastically dominates a $\Bin(d', p)$ random variable for $d' := d(x)-2K$. Thus, we have
  \begin{equation} \label{e:infectedneighbourconditionalapp}
  \begin{split}
      \Prs*{|N(y) \cap A_0| \geq \frac{d(y)}{2}+m \given y \not\in A_0} &\geq \Pr*{\Bin\left(d', p\right)\geq \frac{d(y)}{2}+m}\\
      &\geq \Pr*{\Bin\left(d', p\right)\geq \frac{d'}{2}+m+ \frac{3K}{2}}.
  \end{split}
  \end{equation}
  
From \eqref{e:splitprobyinsapp}, \eqref{e:ynotingivenR}, \eqref{e:infectedneighbourconditionalapp} and \Cref{l:CLT-corollary}, we obtain
\begin{align}\label{e:splitprobyins2app}
    \Prs[\big]{y \in X_1} \geq (1+o(1))(1-p) \cdot \Pr*{\Bin\left(d', p\right)\geq \frac{d'}{2}}.
  \end{align}

We will bound the second term on the right-hand side of \eqref{e:splitprobyins2app} by applying \Cref{l:CLT} for $\Bin(d', p)$. To do that, we need to scale the binomial random variable accordingly. Note that
  \[
    d' \left(\frac{1}{2} - p\right) = \left(1+O\left(\frac{1}{d}\right) \right) \left(\sqrt{\log d} - 2 \lambda \frac{\log \log d}{\sqrt{\log d}}\right)
  \]
  and
  \[
    d' p(1-p)=\left(1+O\left(\frac{\log d(x)}{d(x)}\right)\right) \frac{d(x)}{4}.
  \]
  We may apply \Cref{l:CLT} to bound the right-hand side of \eqref{e:infectedneighbourconditionalapp} and get
  \begin{equation}\label{e:estimatingBinomialapp}
       \Pr*{\Bin\left(d', p\right)\geq \frac{d'}{2}}=(1+o(1)) \Pr{N(0,1)\geq f(d')}
  \end{equation}
  for
  \[
    f(d') := \frac{d'(1/2 - p)}{\sqrt{d' p(1-p)}} = \left(1+O\left(\frac{\log d(x)}{d(x)}\right)\right) \left(\sqrt{\log d(x)} -2 \lambda \frac{\log \log d}{\sqrt{\log d}}\right).
  \]
  For this value of $f$, we may estimate
  \begin{align}
       \Pr{N(0,1)\geq f(d')} &=\frac{1+o(1)}{f(d') \sqrt{2 \pi}} \exp\left(-\frac{1}{2}f(d')^2\right)          \nonumber\\
       &=\frac{1+o(1)}{\sqrt{2\pi \log d}} \exp\left(-\frac{1}{2} \log d + 2 \lambda \log \log d + O(1)\right) \nonumber\\
       &=\Omega\left(d(x)^{-1/2} (\log d(x))^{2 \lambda -1/2}\right).\label{e:estimatingNormalapp}
  \end{align}
  Hence, by \eqref{e:splitprobyins2app}, \eqref{e:estimatingBinomialapp} and \eqref{e:estimatingNormalapp}, we obtain
  \begin{equation}
      \Exs*{|X_1| } =   d(x) \cdot  \Prs[\Big]{y \in X_1} = \Omega\left(d(x)^{1/2} (\log d(x))^{2 \lambda -1/2}\right),
  \end{equation}
  as desired. 
\end{proof}

Note that \Cref{c:Var1N,c:var} still hold and we can directly prove the analogue of \Cref{l:constant}.
Recall that $|X_0| \sim \Bin(d(x),p)$, where 
\[
p = \frac{1}{2}- \frac{1}{2} \sqrt{\frac{\log d}{d}} + \lambda \frac{\log \log d}{\sqrt{d \log d}}
\]
for $\lambda> 1/2$. Taking $\ell: = \Exs*{|X_1|}/2$ we have $\ell = \Omega\left(d(x)^{1/2} (\log d(x))^{2 \lambda -1/2}\right)$ by \Cref{c:exp}, and thus as $\lambda >1/2$
   \begin{equation}
    \Ex{|X_0|} = p d(x) = \frac{d(x)}{2} - \frac{1}{2} \sqrt{d(x)\log d(x)} + \lambda \sqrt{\frac{d(x)}{\log d(x)}} \log \log d(x) > \frac{d(x)}{2} + m - \ell.
  \end{equation}
  Note that this is the analogue of \eqref{e:expE0} as mentioned in \Cref{r:finer1statement} and we can proceed as before to finish the proof.
  

\end{document}